\documentclass[a4paper,10pt]{amsart}
\usepackage{graphicx,amssymb,amstext,amsmath,amsfonts,amsthm}
\usepackage{nicefrac}
\usepackage{geometry}
\usepackage{anysize}
\marginsize{2cm}{2cm}{2cm}{2cm}
\usepackage{mathtools}
\usepackage[normalem]{ulem}
\numberwithin{equation}{section}     
\theoremstyle{plain}
\newtheorem{theorem}{Theorem}[section]
\newtheorem{lemma}{Lemma}[section]
\newtheorem{hypothesis}{Hypothesis}[section]
\newtheorem{corollary}{Corollary}[section]
\newtheorem{proposition}{Proposition}[section]
\newtheorem{remark}{Remark}[section]

\newcommand{\<}{\langle}
\renewcommand{\>}{\rangle}
\newcommand{\lqq}{\leqslant}
\newcommand{\gqq}{\geqslant}
\newcommand{\dv}{\langle\langle}
\newcommand{\dw}{\rangle\rangle}
\newcommand{\ud}{\mathrm{d}}
\newcommand{\pd}{\partial}
\newcommand{\Lap}{\Delta}
\newcommand{\ra}{\rightarrow}
\newcommand{\xX}{\mathcal{X}}
\newcommand{\kK}{\mathcal{K}}
\newcommand{\cL}{\mathcal{L}}
\newcommand{\mM}{\mathcal{M}}
\newcommand{\pP}{\mathcal{P}}
\newcommand{\Wp}{\mathcal{W}_p}
\newcommand{\HL}{{\mathcal{H}}_{\Lambda}}
\newcommand{\Sg}{\mathcal{S}_{\gamma}}
\newcommand{\aA}{\mathcal{A}}
\newcommand{\NN}{\mathbb{N}}
\newcommand{\RR}{\mathbb{R}}
\newcommand{\EE}{\mathbb{E}}
\newcommand{\la}{\lambda}
\newcommand{\ga}{\gamma}
\newcommand{\vr}{\varrho}
\newcommand{\e}{\varepsilon}
\newcommand{\Rea}{\mathsf{Re}}
\newcommand{\ii}{\mathbf{i}}
\usepackage{color}  
\definecolor{DarkBlue}{rgb}{0,0,0.7}
\definecolor{DarkRed}{rgb}{0.95,0,0}

\begin{document}
\title[The cutoff phenomenon for the stochastic heat and wave equation subject to small L\'evy noise]{The cutoff phenomenon for the stochastic heat and wave equation subject to small L\'evy noise}
\author{Gerardo Barrera}
\address{University of Helsinki, Department of Mathematical and Statistical Sciences. Exactum in Kumpula Campus. PL~68, Pietari Kalmin katu 5.
Postal Code: 00560. Helsinki, Finland.}
\email{gerardo.barreravargas@helsinki.fi}
\author{Michael A. H\"ogele}
\address{Departamento de Matem\'aticas, Universidad de los Andes, Bogot\'a, Colombia.}
\email{ma.hoegele@uniandes.edu.co}
\author{Juan Carlos Pardo}
\address{
CIMAT. Jalisco S/N, Valenciana, CP 36240. Guanajuato, Guanajuato, M\'exico.}
\email{jcpardo@cimat.mx }
\keywords{Exponential ergodicity, Error estimates, Heat semigroup, Infinite dimensional Ornstein-Uhlenbeck processes,
L\'evy processes, Stochastic partial differential equations, The cutoff phenomenon, Wasserstein distance} 
\subjclass{60H15; 37A30; 60G51; 47D06}

\maketitle
\begin{abstract}
This article generalizes the small noise cutoff phenomenon obtained recently by Barrera, H\"ogele and Pardo (JSP2021) to the strong solutions of the stochastic heat equation and the damped stochastic wave equation over a bounded domain subject to additive and multiplicative Wiener and L\'evy noises in the Wasserstein distance. 
The methods rely  on the explicit knowledge of the respective eigensystem of the stochastic heat and wave operator and 
the explicit representation of the multiplicative stochastic solution flows in terms of stochastic exponentials. 
\end{abstract}

\section{Introduction}
\markboth{G. Barrera, M. A. H\"ogele and J.C. Pardo}{The cutoff phenomenon for the heat and the wave equation subject to small L\'evy noise}
Stochastic partial differential equations with Gaussian and non-Gaussian L\'evy noise 
are ubiquitous in the applications 
 and have produced a vast literature in mathematics in recent years, standards texts include
\cite{Albeverio,ApplebaumWu,BradleyJurek,BrzezniakZabczyk,
Daprato1982,DapratoGatarek,DaPratoZabczyk1996,DaPratoZabczyk2002,
Dapratobook,Debussche,FuhrmanRockner2000,
LescotRockner2004,Salaheldin,PrevotRockner,PRIOLA,
PriolaZabczyk2010,Riedle2015,Walsh1986}.
This article generalizes the small noise cutoff phenomenon 
from finite-dimensional 
ergodic Ornstein-Uhlenbeck systems 
which was established in \cite{BHPWA} for the Wasserstein distance 
to a class of linear stochastic partial differential equations with Wiener and L\'evy noise. 
 
More precisely, we study the asymptotically abrupt (as $\e\ra 0$) 
ergodic convergence of the strong solutions of the linear stochastic partial differential equations of the following type 
\begin{equation}\label{eq:defsde1}
\ud X^\e_t(h) = A X^\e_t(h) \ud t + \e  \ud \eta_t, \qquad X^\e_0(h) = h\in H, \qquad \e>0,  
\end{equation}
where $A$ is either the Dirichlet Laplacian in an appropriate Hilbert space $H$ or the respective matrix-valued damped wave operator.
We treat additive and multiplicative noise.
For additive noise the process $\eta = L$ with $L = (L_t)_{t\gqq 0} $ is either a $Q$-Brownian motion or a L\'evy process with values in $H$. The generalization of the results in \cite{BHPWA} relies on the explicit knowledge of the eigensystem of $A$. In case of multiplicative noise $\ud \eta_t$ is replaced by $X^\e_t(h) \ud L_t$ in an appropriate sense. For the stochastic heat equation with multiplicative noise we still have a detailed knowledge of the stochastic flow under commutativity assumptions, which allows to establish the cutoff phenomenon. This method seems to break down for the stochastic wave equation with multiplicative noise due to infeasibly strong commutativity requirements on the noise coefficients.  

In a series of articles, \cite{BHPWA,BHPWANO, BHPTV, BJ, BJ1, BP},
the authors have studied the cutoff phenomenon of finite dimensional stochastic differential equations. 
Their setting covers linear or smooth coercive nonlinear dynamical systems close to an asymptotically exponentially stable fixed point subject to different small additive white and red noises in the (renormalized) Wasserstein distance and the total variation distance. 
The idea of this article is to establish the cutoff phenomenon  
to the most famous and elementary class of stochastic partial differential equations, 
that is, the stochastic heat and the stochastic wave equation. 

The concept of the cutoff phenomenon was coined by Aldous and Diaconis in the context of card-shuffling in  \cite{AD}.
{
It roughly states 
the asymptotically abrupt convergence of the marginals to the dynamical equilibrium as a function of the deck size.
As an introduction to
the cutoff phenomenon in discrete time and space we recommend to~\cite{Jonson} and Chapter~18 
in~\cite{LPW}.
Since the seminal paper \cite{AD}, this threshold behavior has been shown to be present in many discrete stochastic systems and most of the results are stated in terms of the total variation distance. However, 
citing \cite{BY1} ``This precision comes at a high technical price and is largely responsible for the variety of treatments found in the literature." 
For standard texts illustrating the mentioned  mathematical richness of cutoff thermalization in \textit{discrete cutoff parameter} we refer} to 
\cite{Al83, AD87, AD, BY1,  BLY06, BY2014, BD92,  BBF08, CSC08, DI87, DIA96, DGM90, DS, La16,  Lachaud2005, LanciaNardiScoppola12, LLP10, LPW, LubetzkySly13, Meliot14, Nestoridi, Pillai, Scarabotti, Trefethenbrothers, Yc99}. 
{Most recent developments in this large field of active research
include topics ranging from
sparse random graphs, sparse random digraphs, sparse Markov chains, sparse bistochastic random matrices, square plaquette model, zero-range process to asymmetric simple exclusion processes (ASEP) 
} are found in \cite{B-HLP19, BCS19, BCS18, BQZ20,   CS21, HS20, LL19}. For a precise review on the results for stochastic differential equations and its embedding in the literature of the discrete cutoff phenomenon we refer to the introductions in \cite{BHPWA, BHPTV,BP}.  

In its weakest version the presence of a \textit{cutoff phenomenon} means the asymptotically abrupt collapse of a suitably renormalized distance $d_\e$ (with $\liminf_{\e \ra 0}\mbox{diam}(d_\e) = \infty$) between the marginal distributions $X^\e_t(h)$ {of \eqref{eq:defsde1}} and the respective invariant limiting distribution $\mu^\e$ along deterministic time scale $t_{\e, h}$ ($t_{\e, h} \ra \infty$ as $\e \ra 0$) in the sense that  
\begin{equation}\label{e: general cutoff}
d_\e (X^\e_{\delta \cdot t_{\e, h}} (h), \mu^\e)\stackrel{\e \ra 0}{\longrightarrow} \begin{cases} \infty & \mbox{ for }\delta>1,\\0 & \mbox{ for }\delta \in (0,1).\end{cases}
\end{equation}
Stronger notions of the cutoff phenomenon come in two formulations. Roughly speaking, the weaker and virtually universally valid \textit{window cutoff phenomenon} describes limits of type \eqref{e: general cutoff}, where the time scale divergence in \eqref{e: general cutoff} $\delta \cdot t_{\e, h} = t_{\e, h}  -(1-\delta) t_{\e, h} \ra \infty$ is of the type 
$t_{\e, h} + \vr$ for any $\rho \ra \infty$ and $\vr \ra - \infty$. 
The \textit{profile cutoff phenomenon} conceptualizes the stronger statement that the limit 
\begin{equation}\label{e: general cutoff0900}
d_\e (X^\e_{t_{\e, h}+\vr} (h), \mu^\e) \stackrel{\e \ra 0}{\longrightarrow} \pP_{h}(\vr)
\end{equation}
exists for any \textit{fixed} $\vr\in \RR$.
It is the main result of \cite{BHPWA} to characterize 
the necessity and sufficiency of the limit $\pP_{h}(\vr)$ for $d_\e = \e^{-1}\cdot \mathcal{W}_p$, $p\gqq 1$, in terms of the normal growth pattern of the (generalized) eigenvectors of $A$, where $\mathcal{W}_p$ is the Wasserstein distance of order $p$. The main step 
consists in the application of a non-standard property of the Wasserstein distance $\mathcal{W}_p$  for $p\gqq 1$ 
established  in Lemma~\ref{lem:properties}.(d) called
\textit{shift linearity}.
The shift linearity 
 allows to determine the  precise shape of the cutoff profile (in case it exists) as 
\[
\pP_{h}(\vr) = K_h \cdot e^{-\Lambda_h \vr},\qquad \vr\in \RR, 
\]
where $K_h, \Lambda_h$ are explicit positive spectral constants associated to the operator $A$ and the initial condition~$h$. In the spirit of \cite{BHPWA} this article combines 
\begin{enumerate}
\item[i)] the precise knowledge of the spectrum and the eigenvectors of the Dirichlet Laplacian 
(and the respective wave operator on the product space of  position times velocity), 
 \item[ii)] the analogous results of Lemma~2.1 in~\cite{BHPWA} correspondingly 
 adapted to the infinite dimensional setting~(Lemma~\ref{l:jaraheat} and Proposition~\ref{prop:casos}) and 
  \item[iii)] the properties of the Wasserstein distance, in particular, the shift linearity (item  d) in Lemma~\ref{lem:properties}). 
 \end{enumerate}
 In case of multiplicative noise, instead of items iii), item i) and ii) are combined with 
\begin{enumerate}
 \item[iv)] the precise representation of the stochastic flow for multiplicative noise and the triviality of the invariant measure~$\mu^\e$.
 \end{enumerate}
As a consequence we establish the following three main results. 
In Theorem~\ref{th:cutoffprofileabstrat} and Corollary~\ref{th:cutoffprofile} we establish the profile cutoff phenomenon for the stochastic heat equation with additive noise, with the help of the self-adjointness and the negativity of the spectrum of the Dirichlet-Laplacian. The profile is calculated explicitly. This is a  generalization of the finite dimensional case. In Theorem~\ref{th:abstwave} it is the precise knowledge of the Dirichlet Laplacian and its extension to the respective wave operator which allows to establish a nontrivial extension of the results of the stochastic damped linear oscillator in \cite{BHPWA} Subsection 4.2. The case of the stochastic heat equation with multiplicative noise is completely new even in finite dimensions. In Theorem~\ref{t:multprofile} and Corollary~\ref{cor:infinitos} the profile cutoff phenomenon is shown for multiplicative $Q$-Brownian motion with the help of the explicit representation of the associated Brownian flow which retains the precise knowledge of the initial value. In Theorem~\ref{t:multprofilelevy} and Corollary~\ref{cor:infinitoslevy} the analogous results are shown for a pure jump L\'evy noise. 
The case of the stochastic wave equation with multiplicative noise seems not to be feasible for cases of interest due to the lack of commutativity.

\bigskip 
\section{Stochastic heat and wave equation with additive noise}\label{s:HB+L}

\subsection{Ornstein-Uhlenbeck process}\label{ss:OUP}
Let $H$ be a separable Hilbert space
and $(S(t))_{t\gqq 0}$ be a $C_0$-semigroup having infinitesimal
generator $A$, and acting in $H$.
That is, $(S(t)h)_{t\gqq 0}$ is given as the {mild} solution of the Cauchy problem
\begin{equation}\label{e:semi}
\left\{
\begin{array}{r@{\;=\;}l}
{\frac{\pd}{\pd t}}S(t)h&A S(t)h\quad  
\textrm{  for any }\quad t\gqq 0,\\
S(0)h & h\in H.
\end{array}
\right.
\end{equation}
Let $\e\in (0,1]$ and $L=(L_t)_{t\gqq 0} $ be a L\'evy process taking values on $H$.
For any $t>0$ the law of $L_t$ is uniquely determined by its characteristics $(b,Q,\nu)$. 
We refer to Section~2 in \cite{Applebaum} for details.
For any $h\in H$ we consider $(X^{\e}_t(h))_{t\gqq 0}$  the unique {mild} solution of the linear stochastic differential equation
\begin{equation}\label{e:OU}
\left\{
\begin{array}{r@{\;=\;}l}
\ud X^{\e}_t(h)&A X^{\e}_t(h)\ud t +
\e \ud L_t\quad  
\textrm{  for any }\quad t\gqq 0,\\
X^{\e}_0(h) & h\in H.
\end{array}
\right.
\end{equation}
The variation of constants formula yields
\begin{equation}\label{e:voc}
X^{\e}_t(h)=S(t)h+\e \int_{0}^{t}S(t-s)\ud L_s.
\end{equation}
In this article we assume that $H$ is infinite-dimensional. The finite dimensional case is completely discussed in \cite{BHPWA}.
\begin{hypothesis}[Asymptotically exponentially stable]\label{h:expsta}
We assume that the  semigroup $(S(t))_{t\gqq 0}$ is asymptotically exponentially stable, i.e. for some $C_*\gqq  1$, $\lambda_*>0$
\begin{equation}\label{e:expsta}
\|S(t)\|\lqq  C_*e^{-\lambda_* t}\quad \textrm{ for all }\quad t\gqq 0,
\end{equation}
where $\|\cdot \|$ denotes the operator norm.
\end{hypothesis}
\noindent Under Hypothesis~\ref{h:expsta}, 
Theorem~6.7 in~\cite{AnnaJ} yields that
\[
\int_{|x|>1}
\log(|x|)\nu(\ud x)<\infty
\]
is a sufficient condition for the existence of a
(unique) invariant measure $\mu^\e$ for the random dynamics given in \eqref{e:OU}.

\bigskip 
\subsection{Wasserstein metric of order $p$-th on $H$}\label{ss:dist}
Since our main results are stated in the Wasserstein distance we gather its most important properties. 
For any two probability distributions $\mu_1$ and $\mu_2$ on $H$
with finite $p$-th moment for some $p>0$, the Wasserstein distance of order $p$ between them is given by 
\begin{equation*}\label{d:war}
\Wp(\mu_1,\mu_2)=
\inf_{\Pi} \left(\int_{H\times H}|u-v|^{p}\Pi(\ud u,\ud v)\right)^{1\wedge (1/p)},
\end{equation*}
where the infimum is taken over all 
couplings $\Pi$ of marginals $\mu_1$ and $\mu_2$, see for instance \cite{Panaretos,Vi09}.
The main properties used in this manuscript are gathered for convenience in the following lemma.
\begin{lemma}[Properties of the Wasserstein distance]\label{lem:properties}\hfill\\
Let $p>0$, $u_1,u_2\in H$, $c\in \RR$ be given and $U_1, U_2$ random vectors in $H$ with finite $p$-th moment. Then we have: 
\begin{itemize}
\item[(a)] The Wasserstein distance defines a metric. 
\item[(b)] Translation invariance: 
$\Wp(u_1+U_1,u_2+U_2)=\Wp(u_1-u_2+U_1,U_2)$.
\item[(c)] Homogeneity: 
\[
\Wp(c\cdot U_1,c\cdot U_2)=
\begin{cases}
|c|\;\Wp(U_1,U_2)&\textrm{ for } p\in [1,\infty),\\
|c|^{p}\;\Wp(U_1,U_2)&\textrm{ for } p\in (0,1).
\end{cases}
\]
\item[(d)] Shift linearity: 
For $p\gqq 1$ it follows 
\begin{equation}\label{eq:shitflinearity}
\Wp(u_1+U_1,U_1)=|u_1|.
\end{equation}
For $p\in(0,1)$ equality \eqref{eq:shitflinearity} is false in general, see Remark~2.4 in~\cite{BHPWA}. However, for any $p>0$ we have the following inequality
\begin{equation}\label{e:cotauni}
\max\{|u_1|^{1\wedge p}-2\EE[|U_1|^{1\wedge p}],0\}\lqq \Wp(u_1+U_1,U_1)\lqq |u_1|^{1\wedge p}.
\end{equation}
\item[{(e)}] Shift continuity:  For any deterministic sequence $(h_n)_{n\in \mathbb{N}}\subset H$ such that $h_n\to h\in H$ as $n \to \infty$ it follows that 
\[
\lim\limits_{n\to \infty}
\Wp(h_n+U_1, U_1) =\Wp(h+U_1, U_1).
\]
\end{itemize}
\end{lemma}
\noindent
{Property (a) is standard, see for instance \cite{Vi09}.}
The properties (b) {and (c)}  
{hold in any Banach space,}
see {for instance} Lemma~2.2 in~\cite{BHPWA} {for $H=\mathbb{R}^d$}. 
The proof of item (d) is given in Appendix~\ref{Ap:A1}. Property {(e)} is a direct consequence of 
\eqref{e:cotauni} and item (b).

\bigskip 
\subsection{Existence of an invariant probability measure $\mu^\e$}\label{ss:ergo}
Since our results are stated in terms of $\mathcal{W}_p$ for some $p>0$, we need the following moment condition.
\begin{hypothesis}[Finite $p$-th moment] \label{h:moment}
There exists $p>0$ such that 
\[
\int_{|x|>1}
|x|^{p}\nu(\ud x)<\infty.
\]
\end{hypothesis}
\noindent We note that 
$\mu^\e\stackrel{d}{=}\e \widetilde{L}_\infty$, where $\widetilde{L}_\infty$ is the limit as $t \to \infty$ (in distribution) of 
\begin{equation}\label{eq:Ltilde}
 \widetilde{L}_t:=\int_{0}^{t}S(t-s)\ud L_s.
\end{equation}
In particular, 
$\mu^1\stackrel{d}{=}\widetilde{L}_\infty$. In the sequel this result is strengthened to the convergence with respect to $\mathcal{W}_p$.
\begin{lemma}[Ergodic convergence in Wasserstein distance]
\label{l:convWp}
Assume Hypotheses~\ref{h:expsta} and \ref{h:moment} are satisfied for some $p>0$.
Then for any $h\in H$ and  $t\gqq 0$ 
\[
\Wp(X^\e_t(h),\mu^\e)\lqq
(C_* e^{-\la_* t}|h|)^{1\wedge p}+
(\e C_* e^{-\la_* t})^{1\wedge p}\int_{H} |h_0|^{1\wedge p}\mu^1(\ud h_0). 
\]
\end{lemma}
\begin{proof}
Let $h,h_0 \in H$ and $t\gqq 0$. 
The natural coupling and \eqref{e:expsta} imply
\begin{equation}
|X^\e_t(h)-X^\e_t(h_0)|=|S(t)(h-h_0)|\lqq \|S(t)\||h-h_0|\lqq C_* e^{-\la_* t}(|h|+|h_0|).
\end{equation}
{Due to \eqref{d:war} the Wasserstein distance is dominated by the $L^2$-distance which yields}
\[
\Wp(X^\e_t(h),X^\e_t(h_0))\lqq 
C^{(1\wedge p)}_* e^{-\la_*(1\wedge p) t} (|h|^{1\wedge p}+|h_0|^{1\wedge p}).
\]
By the Markov property (Theorem~5.1 in \cite{Applebaum}) and disintegration we have the desired result, that is,  
\begin{align*}
\Wp(X^\e_t(h),\mu^\e) &\lqq 
\int_{H} \Wp(X^\e_t(h),X^\e_t(h_0))\mu^\e(\ud h_0)\lqq 
(C_* e^{-\la_* t}|h|)^{1\wedge p}+
(C_* e^{-\la_* t})^{1\wedge p}\int_{H} |h_0|^{1\wedge p}\mu^\e(\ud h_0). 
\end{align*}
\end{proof}

\bigskip 
\subsection{Cutoff inequality for Ornstein-Uhlenbeck processes}\label{ss:cutoffOUP}
In case of additive noise the cutoff phenomenon for $d_\e = \e^{-1}\mathcal{W}_p$, $p>0$, 
relies on the following estimate. 
\begin{lemma}[Cutoff semigroup approximation]\label{lem:inecut}
Assume Hypotheses~\ref{h:expsta} and \ref{h:moment} for some $p>0$.
Then for any 
$h\in H$ and $t\gqq 0$ we have 
\begin{equation}\label{eq:inequalitycutoff}
\left|\frac{\Wp(X^\e_t(h),\mu^\e)}{\e^{1\wedge p}}-\Wp\Big(\frac{S(t)h}{\e}+\widetilde{L}_\infty, \widetilde{L}_\infty\Big)\right|\lqq \Wp(\widetilde{L}_t, \widetilde{L}_\infty),
\end{equation}
where $\widetilde{L}_t$ and $\widetilde{L}_\infty$ are given in \eqref{eq:Ltilde}.
In particular, for any $(t_\e)_{\e>0}$ such that $t_\e \to \infty$ as $\e \to 0$, we have
\begin{equation}\label{e:limsup}
\limsup_{\e\to 0}\frac{\Wp(X^\e_{t_\e}(h),\mu^\e)}{\e^{1\wedge p}}=
\limsup_{\e\to 0}\Wp\Big(\frac{S(t_\e)h}{\e}+\widetilde{L}_\infty, \widetilde{L}_\infty\Big)
\end{equation}
and
\begin{equation}\label{e:liminf}
\liminf_{\e\to 0}\frac{\Wp(X^\e_{t_\e}(h),\mu^\e)}{\e^{1\wedge p}}=
\liminf_{\e\to 0}\Wp\Big(\frac{S(t_\e)h}{\e}+\widetilde{L}_\infty, \widetilde{L}_\infty\Big).
\end{equation}
\end{lemma}
\begin{proof}
Let $h\in H$ and $t\gqq 0$.
For any $p>0$ we have 
\[
\begin{split}
\Wp(X^\e_t(h),\mu^\e)&=\Wp(S(t)h+\e \widetilde{L}_t,\e \widetilde{L}_\infty)\lqq 
\Wp(S(t)h+\e \widetilde{L}_t,S(t)h+\e \widetilde{L}_\infty)+
\Wp(S(t)h+\e \widetilde{L}_\infty,\e \widetilde{L}_\infty)\\
&=
\Wp(\e \widetilde{L}_t,\e \widetilde{L}_\infty)+
\Wp(S(t)h+\e \widetilde{L}_\infty,\e \widetilde{L}_\infty).
\end{split}
\]
Then Lemma~\ref{lem:properties}  
(c) yields 
\begin{equation}\label{e:ine1}
\frac{\Wp(X^\e_t(h),\mu^\e)}{\e^{1\wedge p}}\lqq 
\Wp(\widetilde{L}_t,\widetilde{L}_\infty)+
\Wp\Big(\frac{S(t)h}{\e}+\widetilde{L}_\infty, \widetilde{L}_\infty\Big).
\end{equation}
Analogously, item (a), (b) and (c) of Lemma~\ref{lem:properties} imply
\begin{equation}\label{e:ine2}
\begin{split}
\Wp\left(\frac{S(t)h}{\e}+ \widetilde{L}_\infty,\widetilde{L}_\infty\right)&=\frac{
\Wp(S(t)h+ \e \widetilde{L}_\infty,\e \widetilde{L}_\infty)}{\e^{1\wedge p}}\lqq 
\frac{
\Wp(S(t)h+ \e \widetilde{L}_\infty,S(t)h+ \e \widetilde{L}_t)}{\e^{1\wedge p}}
+\frac{\Wp(S(t)h+\e \widetilde{L}_t,\e \widetilde{L}_{\infty})}{\e^{1\wedge p}}\\
&=
\Wp(\widetilde{L}_\infty, \widetilde{L}_t)
+\frac{\Wp(X^\e_t(h),\mu^\e)}{\e^{1\wedge p}}.
\end{split}
\end{equation}
Combining \eqref{e:ine1} and \eqref{e:ine2} yields 
\eqref{eq:inequalitycutoff}.
We note that 
\eqref{e:limsup} and
\eqref{e:liminf} follow from \eqref{eq:inequalitycutoff} and Lemma~\ref{l:convWp}.
\end{proof}
\noindent
For $p\gqq 1$ the shift linearity simplifies the cutoff phenomenon as follows. 
\begin{corollary}[Cutoff semigroup approximation for $p\gqq 1$]\label{cor:inecut090}
Assume Hypothesis~\ref{h:expsta} and Hypothesis~\ref{h:moment} for some $p\gqq 1$.
For any 
$h\in H$ and $t\gqq 0$
\begin{equation}\label{e:error1}
\left|\frac{\Wp(X^\e_t(h),\mu^\e)}{\e}-\left|\frac{S(t)h}{\e}\right|\right|\lqq \Wp(\widetilde{L}_t, \widetilde{L}_\infty),
\end{equation}
where $\widetilde{L}_t$ and $\widetilde{L}_\infty$ are given in \eqref{eq:Ltilde}.
In particular, for any $(t_\e)_{\e>0}$ such that $t_\e \to \infty$ as $\e \to 0$, we have
\begin{equation}\label{e:limsup090}
\limsup_{\e\to 0}\frac{\Wp(X^\e_{t_\e}(h),\mu^\e)}{\e}=
\limsup_{\e\to 0}\left|\frac{S(t_\e)h}{\e}\right|
\qquad \mbox{ and }\qquad \liminf_{\e\to 0}\frac{\Wp(X^\e_{t_\e}(h),\mu^\e)}{\e}=
\liminf_{\e\to 0}\left|\frac{S(t_\e)h}{\e}\right|.
\end{equation}
\end{corollary}
\begin{proof}
It follows directly from inequality \eqref{eq:inequalitycutoff} in Lemma~\ref{lem:inecut} and the shift linearity for $p\gqq 1$ given in Lemma~\ref{lem:properties}(d).
\end{proof}
\noindent By reparametrization the small noise cutoff phenomenon for fixed initial data is equivalent to the cutoff phenomenon for initial data and constant noise intensity.
\begin{corollary}[Large initial data]\label{cor:largedata}
Assume Hypothesis~\ref{h:expsta} and Hypothesis~\ref{h:moment} for some $p>0$.
Then for any $h\in H$ and  $t\gqq 0$ it follows that
\begin{equation}
\frac{\Wp(X^\e_{t}(h),\mu^\e)}{{\e^{1\wedge p}}}=\Wp(X^{1}_{t}(h/\e),\mu^1)\quad \textrm{ for any }\quad \e\in(0,1).
\end{equation}
\end{corollary}
\begin{proof}
It follows directly  from the homogeneity given in item (c) of Lemma~\ref{lem:properties} and the fact that
$\mu^\e\stackrel{d}{=}\e \widetilde{L}_\infty$, 
$\mu^1\stackrel{d}{=}\widetilde{L}_\infty$, where $\widetilde{L}_\infty$ is the limit as $t \to \infty$ (in distribution) of \eqref{eq:Ltilde}.
\end{proof}

\bigskip 
\section{Profile cutoff for the heat equation with additive noise}\label{s:HA}
\noindent
In this section we establish the profile cutoff phenomenon for a slightly generalized heat equation \eqref{e:OU}  subject to small additive  noise. For the linear operator $A$ we assume the following hypothesis.
\begin{hypothesis}\label{h:sadj}
We assume that
$A$ is a self-adjoint, strictly negative operator with domain $D(A)\subset H$
such that there is a fixed orthonormal basis 
{$(e_n)_{n\in \NN_0}$} in $H$ verifying: {$(e_n)_{n\in \NN_0}\subset D(A)$},
$A e_n=-\lambda_n e_n$, 
with $0<\la_0\lqq \dots\lqq \la_n\lqq \la_{n+1}$ for all $n$
such that
 $\lambda_n \rightarrow \infty$.
\end{hypothesis}
\noindent
The standard example we have in mind is the standard Dirichlet Laplacian {$\Delta = \sum_{i=1}^d \frac{\partial^2}{\partial x^2_i}$} in $L^2(\mathcal{O})$ for $\mathcal{O}$ being an open bounded connected domain $\mathcal{O}\subset\RR^d$ with Lipschitz boundary, see Chapter~10 in \cite{Brezis}.
The second  example in mind is the Stokes operator, see for instance Section~4.5 on thermohydraulics in \cite{TEMANBOOK}.
The following lemma is the infinite dimensional version of Lemma~2.1 in~\cite{BHPWA} for the heat semigroup $S$. 
\begin{lemma}\label{l:jaraheat}
Assume Hypothesis~\ref{h:sadj}.
Then for each $h\in H$, $h\neq 0$
there exist $N_h\in \NN_0$  and $v_h\in H$, $v_h\neq 0$, 
such that 
\begin{equation}\label{e:jaraheat}
\lim\limits_{t\to \infty}|e^{\la_{{N}_h} t}S(t)h-v_h|=0.
\end{equation}
\end{lemma}
\begin{proof}
By Hypothesis~\eqref{h:sadj} we have for any $h\in H$ and $t\gqq 0$
\[
S(t)h=\sum\limits_{k=0}^{\infty} e^{-\la_k t}\<e_k,h\>e_k.
\]
For $h\neq 0$ we define 
\begin{equation}\label{eq:Nh}
\begin{split}
&\kK_h=\{k\in \NN_0: \<e_k,h\>\neq 0\}\neq \emptyset,\quad
N_h:=\min(\kK_h),\\
& \mathcal{M}_h=\{k\in \kK_h: \la_k=\la_{N_h}\} \quad \textrm{ and }\quad N^*_h=\min (\kK_h\setminus \mM_h),
\end{split}
\end{equation}
and obtain the representation 
\[
\begin{split}
e^{t \la_{N_h}}S(t)h&=\sum\limits_{k\in \kK_h} e^{(\la_{N_h}-\la_k) t}\<e_k,h\>e_k=
\sum\limits_{k\in \mM_h}
\<e_{k},h\>e_{k}+
\sum\limits_{k\in \kK_h\setminus \mM_h} e^{(\la_{N_h}-\la_k) t}\<e_k,h\>e_k.
\end{split}
\]
Setting 
\begin{equation}\label{eq:notvh}
v_h:=\sum\limits_{k\in \mM_h}
\<e_{k},h\>e_{k}
\end{equation}
the orthogonality of the family  $(e_n)_{n\in \mathbb{N}_0}$ yields 
\begin{equation}\label{eq:vh}
|v_h|^2=\sum\limits_{k\in \mM_h}
\<e_{k},h\>^2
\begin{cases}
\gqq \<e_{N_h},h\>^2>0,\\[3mm]
\lqq |h|^2<\infty,
\end{cases}
\end{equation}
and conclude \eqref{e:jaraheat} by 
\begin{equation}\label{eq:errordif}
\begin{split}
|e^{t \la_{N_h}}S(t)h-v_h|^2&= 
\sum\limits_{k\in \kK_h\setminus 
\mM_h} e^{2(\la_{N_h}-\la_{k}) t} |\<e_k,h\>|^2\lqq e^{2(\la_{N_h}-\la_{N^*_{h}}) t} |h|^2\ra 0, 
\quad \mbox{ as }\quad t\ra\infty. 
\end{split}
\end{equation}
\end{proof}
\begin{remark} In the case of 
$0<\la_0< \dots< \la_n< \la_{n+1}$ for all $n$
such that
 $\lambda_n \rightarrow \infty$, the norm is given by $|v_h|^2=\<e_k,h\>^2$, where
 $k$ is the smallest non-negative integer such that $\<e_k,h\>\neq 0$. 
\end{remark}
\noindent We state the first main result which provides an abstract cutoff profile for any $p>0$.
\begin{theorem}[Abstract heat profile cutoff phenomenon for any $p>0$]\label{th:cutoffprofileabstrat}
Let Hypotheses~\ref{h:moment} and \ref{h:sadj} be satisfied for some $p>0$. 
We keep the notation of Lemma~\ref{l:jaraheat}.
Then for $h\in H$, $h\neq 0$ there exists a unique cutoff profile $\pP_{p,h}:\RR \to (0,\infty)$ given by 
\begin{equation}\label{e:profile090}
\pP_{p,h}(\vr):=
\lim\limits_{\e \to 0}\frac{\Wp(X^\e_{t_\e+\vr}(h),\mu^\e)}{\e^{1\wedge p}}=\Wp(e^{-\vr \la_{N_h}}v_h+\widetilde{L}_\infty,\widetilde{L}_\infty),
\end{equation}
where $\widetilde{L}_\infty$ is given in \eqref{eq:Ltilde},
\begin{equation}\label{e:scale090}
t_{\e}=\frac{1}{\la_{N_h}}|\ln(\e)|, \quad \e\in (0,1).
\end{equation}
and $v_h$ is given explicitly in \eqref{eq:notvh}.
\end{theorem}
\noindent We stress that in general the right-hand side of \eqref{e:profile090} is abstract.
For $H=\RR^d$ the result is given in Theorem~3.3 in~\cite{BHPWA}.
\begin{proof}
Let $p>0$, $h\in H$, $h\neq 0$ and $t\gqq 0$. 
For any $(\mathfrak{t}_\e)_{\e>0}$ such that $\mathfrak{t}_\e\to \infty$ as $\e\to 0$, the limits
\eqref{e:limsup} and \eqref{e:liminf} given in Lemma~\ref{lem:inecut} yield that
\begin{equation}\label{e:limitexist090}
\lim_{\e\to 0}\frac{\Wp(X^\e_{\mathfrak{t}_\e}(h),\mu^\e)}{\e^{1\wedge p}}\quad \textrm{exists if and only if}\quad
\lim_{\e\to 0}\Wp\Big(\frac{S(\mathfrak{t}_\e)h}{\e}+\widetilde{L}_\infty, \widetilde{L}_\infty\Big)\quad \textrm{exists}
\end{equation}
and both limits coincide.
Note that the choice of $(t_\e)_{\e\in (0,1)}$ implies $\frac{e^{-\la_{N_h} t_\e}}{\e}=1$.
We observe that
\begin{equation}\label{e:desigual090}
\begin{split}
\left|\frac{S(t_\e+\vr)h}{\e}-{e^{-\la_{N_h} \vr}v_h}\right|&=
\left|\frac{S(t_\e+\vr)h}{\e}-\frac{e^{-\la_{N_h} (t_\e+\vr)}v_h}{\e}\right|=\frac{e^{-\la_{N_h} (t_\e+\vr)}}{\e}\left|e^{\la_{N_h} (t_\e+\vr)}S(t_\e+\vr)h-v_h \right|\\
&=e^{-\vr\la_{N_h}}\left|e^{\la_{N_h} (t_\e+\vr)}S(t_\e+\vr)h-v_h \right|.
\end{split}
\end{equation}
By Lemma~\ref{l:jaraheat} and since $t_\e\to \infty$ as $\e\to 0$, we have that the  right-hand side of \eqref{e:desigual090} tends to zero as $\e\to 0$ and hence $\frac{S(t_\e+\vr)h}{\e}\to {e^{-\la_{N_h} \vr}v_h}\in H$ as $\e \to 0$.
Combining \eqref{e:limitexist090} and the preceding limit with the help of item {(e)} in Lemma~\ref{lem:properties}
yields \eqref{e:profile090}.
\end{proof}
\noindent For $p\gqq 1$ the cutoff profile can be calculated explicitly as follows. 
\begin{corollary}[Heat cutoff profile for $p\gqq 1$]\label{th:cutoffprofile}
Let Hypotheses~\ref{h:moment} and \ref{h:sadj} be satisfied for some $p\gqq 1$.  We keep the notation of Lemma~\ref{l:jaraheat}.
Then for  $h\in H$, $h\neq 0$ there exists a unique cutoff profile $\pP_{h}:\RR \to (0,\infty)$ given by 
\begin{equation}\label{e:profile}
\pP_{h}(\vr):=
\lim\limits_{\e \to 0}\frac{\Wp(X^\e_{t_\e+\vr}(h),\mu^\e)}{\e}=
e^{-\vr \la_{N_h}}|v_h|, 
\end{equation}
where $t_{\e}$ is given in \eqref{e:scale090} and $|v_h|$ is given explicitly in \eqref{eq:vh}.
\end{corollary}
\begin{proof}
\noindent The proof is a direct consequence of the inequality in item (d) of Lemma~\ref{lem:properties} applied to right-hand side of \eqref{e:profile090} in Theorem~\ref{th:cutoffprofileabstrat}.
\end{proof}
\noindent The profile error can be quantified explicitly.  
\begin{corollary}[Profile error estimate for any $p\gqq 1$]\label{cor:errorheat}
Let the assumptions of Theorem~\ref{th:cutoffprofile} be satisfied for some $p\gqq 1$.
Then for all $\vr\in \RR$ there exists a positive constant $K:=K(\vr,C_*,\lambda_*,\mathbb{E}[|\widetilde{L}_\infty|])$ such that for all $\e\in (0,1)$ with $t_\e + \rho\gqq 0$ we have 
\begin{align}
\left|
\frac{\Wp(X^\e_{t_\e+\vr}(h),\mu^\e)}{\e}-\pP_{h}(\vr)\right|\lqq K\left( \e^{\nicefrac{\la_*}{\la_{N^*_{h}}}}+
\e^{\big(1-\nicefrac{\la_{N_{h}}}{\la_{N^*_{h}}}\big)}|h|\right)
\end{align}
for any $h\in H$, where $C_*$ and $\lambda_*$ are given in \eqref{e:expsta}, and $N_{h}$, $N^*_{h}$ are defined in \eqref{eq:Nh}.
\end{corollary}
\begin{proof}
The triangle inequality combined with  \eqref{e:error1} in Corollary~\ref{cor:inecut090}, \eqref{e:desigual090} and  Lemma~\ref{l:convWp}  implies
\begin{align*}
\left|
\frac{\Wp(X^\e_{t_\e+\vr}(h),\mu^\e)}{\e}-\pP_{h}(\vr)\right|
&\lqq 
\left|
\frac{\Wp(X^\e_{t_\e+\vr}(h),\mu^\e)}{\e}-\left|\frac{S(t_\e+\vr)h}{\e}\right|\right|+\left|\frac{S(t_\e+\vr)h}{\e}-{e^{-\la_{N_h} \vr}v_h}\right|\\
&\lqq \Wp(\widetilde{L}_{t_\e+\vr}, \widetilde{L}_\infty)+e^{-\vr\la_{N_h}}\left|e^{\la_{N_h} (t_\e+\vr)}S(t_\e+\vr)h-v_h \right|\\
&\lqq 
C_* e^{-\la_* (t_\e+\vr)}\mathbb{E}[|\widetilde{L}_\infty|]+
e^{(\la_{N_h}-\la_{N^*_{h}}) (t_\e+\vr)} |h|\\
&= \e^{\nicefrac{\la_*}{\la_{N^*_{h}}}}\cdot e^{-\la_*\vr}C_*\mathbb{E}[|\widetilde{L}_\infty|]+
\e^{\big(1-\nicefrac{\la_{N_{h}}}{\la_{N^*_{h}}}\big)}\cdot e^{(\la_{N_h}-\la_{N^*_{h}}) \vr} |h|.
\end{align*}
This completes the proof. 
\end{proof}
\noindent
The following corollary shows that
the profile cutoff phenomenon is the strongest notion of cutoff in this article.
\begin{corollary}[Simple cutoff phenomenon]\label{cor:simpleheat}
Let the assumptions of Theorem~\ref{th:cutoffprofile} be satisfied for some $p>0$.
Then for any $h\in H$ with $h\neq 0$ it follows
\[
\lim\limits_{\e\to 0}
\frac{\Wp(X^\e_{\delta \cdot t_\e}(h),\mu^\e)}{\e^{1\wedge p}}=
\begin{cases}
\infty & \textrm{ for } \delta\in (0,1),\\
0 & \textrm{ for } \delta\in (1,\infty),
\end{cases}
\]
where $t_\e$ is given in \eqref{e:scale090}.
\end{corollary}
\noindent The proof is virtually identical to the proof of Corollary~3.2 in~\cite{BHPWA} with $e^{-\mathcal{Q}t}$ being replaced by $S(t)$.
\noindent By reparametrization the small noise cutoff phenomenon for fixed initial data is equivalent to the cutoff phenomenon for initial data and constant noise intensity.
\begin{corollary}[Large initial data for the heat equation]\label{cor:largedataheat}
Let the assumptions of Theorem~\ref{th:cutoffprofileabstrat} for some $p>0$. 
Then for any $h\in H$, $h\neq 0$, it follows 
\begin{equation*}
\lim\limits_{\e \to 0}{\Wp(X^1_{t_\e+\vr}(h/\e),\mu^1)}=\pP_{p,h}(\vr).
\end{equation*}
The analogues of Corollary~\ref{th:cutoffprofile},  Corollary~\ref{cor:errorheat} and Corollary~\ref{cor:simpleheat} are valid when 
$\frac{\Wp(X^\e_{t}(h),\mu^\e)}{{\e^{1\wedge p}}}$ is replaced by $\Wp(X^{1}_{t}(h/\e),\mu^1)$.
\end{corollary}
\begin{proof}
It follows directly from Corollary~\ref{cor:largedata}.
\end{proof}

\bigskip 
\section{Window and profile cutoff for the wave equation with additive noise}\label{s:WA}

\noindent
In this section we establish the profile and window cutoff phenomenon for the wave equation \eqref{e:OU}  subject to small additive  noise. The linear operator $A$ in \eqref{e:OU} is given by
\[
A:=\begin{pmatrix}
0  & \iota\\
\Lap & -\ga \iota
\end{pmatrix}
\]
for some positive $\gamma$.
The operator $\Delta$ denotes the Dirichlet Laplacian
under the following setting.

\begin{hypothesis}\label{h:sadjwave}
Let $H=L^2(\mathcal{O})$ for an open bounded connected domain $\mathcal{O}\subset \RR^d$ with Lipschitz boundaries. 
We assume that
$\Delta$ is a self-adjoint, strictly negative operator with domain $D(\Delta)\subset H$
such that there is a fixed orthonormal basis $(e_n)_{n\in \NN}$ in $H$ verifying: $(e_n)_{n\in \NN}\subset D(\Delta)$,
$\Delta e_n=-\lambda_n e_n$
with Dirichlet boundary conditions,
 that is,
 $e_n\big|_{\partial\mathcal{O}}=0$ and a simple growing point spectrum 
$0<\la_0< \dots< \la_n< \la_{n+1}$ for all $n$
such that
 $\lambda_n \rightarrow \infty$.
\end{hypothesis}
\noindent
The operator $\iota$ denotes the natural embedding operator of $D((-\Delta)^{1/2})$ in $H$.
\noindent
In this setting, we consider 
the Hilbert space $\mathcal{H}=D((-\Delta)^{1/2}) \times H$,  with the inner product
\begin{equation}
\<\<z,\tilde{z}\>\>
=\<u,\widetilde{u}\>_{D((-\Delta)^{1/2})}+\<h,\widetilde{h}\>,
\end{equation}
where $z=(u,h),\tilde{z}=(\tilde{u},\tilde{h})\in \mathcal{H}$, and
\[
\<u, \widetilde{u}\>_{D((-\Lap)^{1/2})}
=\<u,\widetilde{u}\>+\<\nabla u,\nabla \widetilde{u}\>_{H^d}.
\]
Let $(\mathcal{S}_\gamma(t)z)_{t\gqq 0}$ be the solution of the Cauchy problem
\begin{equation}\label{e:semiwave}
\left\{
\begin{array}{r@{\;=\;}l}
\frac{\pd}{\pd t}\mathcal{S}_\gamma(t)z& A \mathcal{S}_\gamma(t)z\quad  
\textrm{  for any }\quad t\gqq 0,\\
\mathcal{S}_\gamma(0)z & z\in \mathcal{H}
\end{array}
\right.
\end{equation}
and consider the stochastic linear damped wave equation
\begin{equation}\label{e:wave}
\left\{
\begin{array}{r@{\;=\;}l}
\ud X^{\e}_t(z)&A X^{\e}_t(z)\ud t +
\e \ud \cL_t\quad  
\textrm{  for any }\quad t\gqq 0,\\
X^{\e}_0(z) & z,
\end{array}
\right.
\end{equation} 
where
\[
\cL_t=\begin{pmatrix}
0  \\
L_t
\end{pmatrix}, \quad t\gqq 0,
\]
and $L=(L_t)_{t\gqq 0} $ is a L\'evy process taking values on $H$.
The domain of $A$ is given by $D(\Delta)\times D((-\Delta)^{1/2})$.
Since $\gamma>0$, Proposition~3.5 (3) in~\cite{Pritchard81} yields that the semigroup $\Sg$ is asymptotically exponentially stable, that is, there are $C_*, \lambda_*>0$ such that 
\begin{equation}\label{eq:estrella}
\|\Sg(t)\| \lqq C_* e^{-\lambda_* t}\quad \textrm{ for all } \quad t\gqq 0. 
\end{equation}
This implies the existence of a unique invariant probability measure $\mu^\e$ for the dynamics \eqref{e:wave}.

\bigskip 
\subsection{{Explicit computations of the eigensystem}}
In the sequel, we calculate the spectrum and the eigenfunctions  of $A$ in terms of the well-known spectrum and the eigenfunctions of $-\Delta$.
Denote the spectrum of $-\Lap$ by $0< \la_1< \cdots <\la_n< \cdots$.
Formally, the characteristic equations of $A$ are given by
\begin{equation}\label{eq:careq}
\omega^2+\gamma \omega+\la=0,
\end{equation}
where $\la\in \{\la_j:j\in \mathbb{N}\}$. 
This can be made precise in terms of spectral calculus.
For this sake, we determine the eigenfunctions of
$A$. 
The point spectrum of $A$ (conveniently labelled)  is given by
\begin{equation}
\begin{cases}
\hat{\omega}_{\pm k}:=-\frac{\ga}{2}\pm \frac{\sqrt{\gamma^2-4\la_k}}{2}\quad & \textrm{ for  $k\in \mathbb{N}$ such that } \gamma^2>4\la_k,
\\
\omega_{0}:=-\frac{\gamma}{2} \quad&
\textrm{ for  $k\in \mathbb{N}$ such that }
\gamma^2=4\la_k,\\
{\omega}_{\pm k}:=
-\frac{\ga}{2}\pm \ii \frac{\sqrt{4\la_k-\gamma^2}}{2}\quad & \textrm{ for  $k\in \mathbb{N}$ such that }
\gamma^2<4\la_k.
\end{cases}
\end{equation}
We point out that the set $\{k\in \mathbb{N}:\gamma^2\gqq 4\la_k\}$ may be empty (subcritical damping).
\begin{hypothesis}\label{h:nonresonance}
We assume the non-resonance condition $\gamma^2\neq 4\la_k$ for all $k\in \NN$. 
\end{hypothesis}
\noindent
The eigenfunctions of $A$ can be calculated explicitly in terms of the eigenfunctions  $(e_k)_{k\in \mathbb{N}}$  of $-\Delta$, where 
 $\Lap e_k=-\la_k e_k$. 
That is,
\begin{equation}
\begin{cases}
\hat{v}_{\pm k}:=\begin{pmatrix}
e_k\\
\hat{\omega}_{\pm k} e_k
\end{pmatrix}\quad & \textrm{ for  $k\in \mathbb{N}$ such that } \gamma^2>4\la_k,
\\
{v}_{\pm k}:=
\begin{pmatrix}
e_k\\
\omega_{\pm k} e_k
\end{pmatrix}
\quad & \textrm{ for  $k\in \mathbb{N}$ such that }
\gamma^2<4\la_k.
\end{cases}
\end{equation}
\bigskip 
\subsection{Explicit computations of the inner products}
Since $e_k\in D(\Delta)$ for all $k\in \mathbb{N}$, the Dirichlet boundary conditions in Hypothesis~\ref{h:sadjwave} and Green's formula
yield
\[
\<\nabla e_j,\nabla e_k\>_{H^d}=\<e_j,-\Delta e_k\>=\<e_j,\lambda_k e_k\>=\lambda_k \<e_{j},e_{k}\>=\lambda_k\delta_{j,k} \quad \textrm{ for all } j,k\in \mathbb{N},
\]
where $\delta_{j,k}$ is Kronecker's delta.
Hence for $j,k\in \mathbb{Z}\setminus\{0\}$ such that 
$\gamma^2<4\la_{|j|}$ and $\gamma^2<4\la_{|k|}$, it follows that
\begin{equation}\label{eq:eigenbase}
\begin{split}
\dv v_j,v_{k}\dw&=
\<e_{|j|},e_{|k|}\>+
\<\nabla e_{|j|},\nabla e_{|k|}\>_{H^d}+\<\omega_{j} e_{|j|},\omega_{k} e_{|k|}\>
=
\<e_{|j|},e_{|k|}\>+
\lambda_{|k|}\<e_{|j|},e_{|k|}\>+\omega_{j}\bar \omega_{k}\< e_{|j|}, e_{|k|}\>\\
&=(1+\la_{|k|}+|\omega_k|^2\cdot\delta_{j,k}+\omega^2_{-k}\cdot\delta_{j,-k})\cdot\delta_{|j|,|k|},
\end{split}
\end{equation}
where in an abuse of notation the inner product 
$\<\omega_j e_j,\omega_k e_k\>$ denotes the standard sesqui-linear complexification of the inner product $\<,\>$ in the real vector space $H$.
In particular,
\begin{equation} \label{eq:innerproducts}
\begin{split}
\dv v_k,v_{k}\dw 
=1+\la_{|k|}+|\omega_k|^2,\quad
\dv v_{-k},v_{k}\dw 
=1+\la_{|k|}+\omega^2_{-k}\quad \textrm{ and }\quad
\dv v_{k},v_{-k}\dw 
=1+\la_{|k|}+\omega^2_{k}.
\end{split}
\end{equation}
For convenience denote $\mathbb{Z}_*=\mathbb{Z}\setminus\{0\}$.
Analogously to the previous calculations, we have the following cases:
\begin{itemize}
\item[i)]
For $j,k\in \mathbb{Z}_*$ such that 
$\gamma^2>4\la_{|j|}$ and $\gamma^2>4\la_{|k|}$
\begin{equation}\label{eq:eigenbase1}
\begin{split}
\dv \hat v_j,\hat v_{k}\dw&=
\<e_{|j|},e_{|k|}\>+
\<\nabla e_{|j|},\nabla e_{|k|}\>_{H^d}+\<\hat \omega_j e_{|j|},\hat \omega_k e_{|k|}\>=(1+\la_{|k|}+\hat \omega^2_k\cdot 
\delta_{j,k}+\hat \omega_{-k}\hat \omega_{k}\cdot
\delta_{j,-k}
)\cdot\delta_{|j|,|k|}.
\end{split}
\end{equation}
In particular,
$
\dv \hat v_k,\hat v_{k}\dw 
=1+\la_{|k|}+\hat \omega^2_k\quad \textrm{ and }\quad
\dv \hat v_{-k},\hat v_{k}\dw=\dv \hat v_{k},\hat v_{-k}\dw
=1+\la_{|k|}+\hat \omega_{-k} \hat\omega_{k}.
$\\
\item[ii)] For $j,k\in \mathbb{Z}_*$ such that 
$\gamma^2>4\la_{|j|}$ and $\gamma^2<4\la_{|k|}$
\begin{equation}\label{eq:eigenbase3}
\begin{split}
\dv \hat v_j,v_{k}\dw&=
\<e_{|j|},e_{|k|}\>+
\<\nabla e_{|j|},\nabla e_{|k|}\>_{H^d}+\<\hat \omega_j e_{|j|},\omega_k e_{|k|}\>
=0,\quad \mbox{ due to }\quad |j|\neq |k|.
\end{split}
\end{equation}
\end{itemize}
Let $\Lambda=(\lambda_j)_{j\in \mathbb{N}}$ be the eigenvalues of $-\Delta$. We recall Hypothesis~\ref{h:nonresonance} and define the unique index 
\begin{equation}\label{e:jotaestrella}
j_* := \begin{cases} 
\min\{j\in \NN~|~ \gamma^2 \in (4\la_{j}, 4 \la_{j+1})\} & \mbox{ if } \gamma^2 > 4 \la_1,\\ 
0 & \mbox{ if } \gamma^2 < 4\la_1.\end{cases} 
\end{equation}
We define the $\Lambda$-weighted subspace of $\mathcal{H}$ by
\[
{\mathcal{H}}_{\Lambda}:=\left\{z\in \mathcal{H}:z=\sum_{\substack{j\in \mathbb{Z}_*\\ |j|\lqq j_*}}a_j \hat{v}_j+\sum_{\substack{j\in \mathbb{Z}_*\\ |j|> j_*}}b_j v_{j}\quad \textrm{and}\quad\sum_{\substack{j\in \mathbb{Z}_*\\ |j|> j_*}}|b_j|^2\la_{|j|}<\infty\right\}.
\] 
The following proposition is the infinite dimensional version of Lemma~2.1 in \cite{BHPWA} for the wave semigroup $\Sg$ with initial values in ${\mathcal{H}}_{\Lambda}$.
\begin{proposition}\label{prop:casos}
Assume Hypotheses~\ref{h:sadjwave} and \ref{h:nonresonance}.
Then for any given $z\in \HL$, $z\neq 0$, we have the following statements: 
\begin{enumerate}
\item[I)] \textbf{Overdamping.}
If $(\gamma / 2)^2 > \la_1$, there exist 
$\omega_*:=\omega_*(z)>0$ and $v:=v(z)\neq 0$ such that
\begin{equation}\label{e:wavejaraA}
\lim\limits_{t\to \infty}e^{\omega_* t}\mathcal{S}_\gamma(t)z=v(z).
\end{equation}
\item[II)] 
\textbf{Subcritical damping.}
If $(\gamma / 2)^2 < \la_1$, $\omega_*:=\gamma/2>0$, there is $v(t,z)\in \mathcal{H}$ such that
\begin{equation}\label{e:wavejaraB}
\lim\limits_{t\to \infty}
|e^{\omega_* t}\mathcal{S}_\gamma(t)z-v(t,z)|=0
\end{equation}
and it follows 
\begin{equation}\label{eq:cotasinfsup}
0<\liminf\limits_{t\to \infty}|v(t,z)|\lqq 
\limsup\limits_{t\to \infty}|v(t,z)|<\infty.
\end{equation}
\end{enumerate}
\end{proposition}
\noindent The proof is given by Lemma~\ref{lem:spectral} (Case I)) and Lemma~\ref{lem:spectral2} (Case II)) after Theorem~\ref{th:abstwave}.
\begin{remark}
Note that formally the limit \eqref{e:wavejaraA} in I) implies a limit of type \eqref{e:wavejaraB}. However, in Case I) we obtain the stronger profile cutoff phenomenon, while in Case II) we still obtain the window cutoff phenomenon.
\end{remark}
\noindent
The following theorem shows the cutoff phenomenon for the overdamped and subcritically damped wave equation 
in the Wasserstein distance on $\mathcal{H}$
with additive noise for any $p>0$. 
\begin{theorem}[Abstract profile and window cutoff phenomenon for any $p>0$]\label{th:abstwave}
Assume Hypotheses~\ref{h:sadjwave}, \ref{h:nonresonance} and \ref{h:moment} for some $p>0$.
Then 
we have the following results in the respective cases I) and II) of Proposition~\ref{prop:casos} (with their respective notation).
\begin{enumerate}
\item[i)] Let $z\in \HL$, $z\neq 0$, $\omega_*:=\omega_*(z)$, $v:=v(z)$ given in item I) of Proposition~\ref{prop:casos} and
\begin{equation}\label{eq:timereal}
t_\e:=t_\e(z)=\frac{1}{\omega_*}|\ln(\e)|,\qquad \e>0. 
\end{equation}
Then it holds the following profile cutoff phenomenon with the 
(abstract) cutoff profile $\mathcal{P}_{p,z}(\vr)$: 
\begin{equation}\label{eq:hjk}
\lim\limits_{\e\to 0}
\frac{\Wp(X^\e_{t_\e+\vr}(z),\mu^\e)}{\e^{1\wedge p}}=\Wp(e^{-\omega_* \vr }v+\widetilde{\mathcal{L}}_\infty,\widetilde{\mathcal{L}}_\infty)=:\mathcal{P}_{p,z}(\vr)\quad \textrm{ for any} \quad \vr\in \mathbb{R}
\end{equation}
and
\[
\widetilde{\mathcal{L}}_\infty\stackrel{d}{=}\lim\limits_{t\to \infty}\int_{0}^{t}\mathcal{S}_\gamma(t-s)\ud \mathcal{L}(s).
\]
\item[ii)] 
Let $z\in \HL$, $z\neq 0$, $\omega_*:=\gamma/2$ given in item II) of Proposition~\ref{prop:casos} and 
\begin{equation}\label{eq:timecomplejo}
t_\e=\frac{2}{\gamma}|\ln(\e)|, \qquad \e>0. 
\end{equation}
Then we have the following window cutoff phenomenon 
\begin{equation}\label{eq:inftylimitzero}
\lim\limits_{\vr\to -\infty}\liminf\limits_{\e\to 0}
\frac{\Wp(X^\e_{t_\e+\vr}(z),\mu^\e)}{\e^{1\wedge p}}=\infty
\qquad \mbox{ and }\qquad \lim\limits_{\vr\to \infty}\limsup\limits_{\e\to 0}
\frac{\Wp(X^\e_{t_\e+\vr}(z),\mu^\e)}{\e^{1\wedge p}}=0.
\end{equation}
\end{enumerate}
\end{theorem}
\begin{proof}[Proof of Theorem~\ref{th:abstwave}]
Let $p>0$, $z\in \HL$, $z\neq 0$, $t\gqq 0$. \\
\noindent \textbf{Proof of case i).}
{
For any $(s_\e)_{\e>0}$ such that $s_\e\to \infty$ as $\e\to 0$, limits
\eqref{e:limsup} and \eqref{e:liminf} given in Lemma~\ref{lem:inecut} yield that
\begin{equation}\label{e:limitexistz87}
\lim_{\e\to 0}\frac{\Wp(X^\e_{s_\e}(z),\mu^\e)}{\e^{1\wedge p}}\quad \textrm{exists if and only if}\quad
\lim_{\e\to 0}\Wp\Big(\frac{\Sg(s_\e)z}{\e}+\widetilde{\mathcal{L}}_\infty, \widetilde{\mathcal{L}}_\infty\Big)\quad \textrm{exists}
\end{equation}
and both limits coincide.
}
{By \eqref{eq:timereal} we observe} 
$e^{-\omega_* t_\e}=\e$ and
\begin{equation}\label{eq:trianwave}
\begin{split}
\left|\frac{\Sg(t_\e+\vr)z}{\e}- e^{-\omega_*\vr}v(z)\right|= e^{-\omega_*\vr}
\left|e^{\omega_*(t_\e+\vr)}\Sg(t_\e+\vr)z-v(z)\right|.
\end{split}
\end{equation}
Proposition~\ref{prop:casos} item I) with the help of \eqref{eq:trianwave} 
implies 
\begin{equation}\label{eq:limitwave1}
\lim\limits_{\e\to 0}\left|\frac{\Sg(t_\e+\vr)z}{\e}-e^{-\omega_*\vr}v(z)\right|=0.
\end{equation}
Combining \eqref{e:limitexistz87} and \eqref{eq:limitwave1} with the help of item {(e)} in Lemma~\ref{lem:properties}
yields \eqref{eq:hjk}.
\hfill
\\

\noindent
\textbf{Proof of case ii).} 
{
For any $(s_\e)_{\e>0}$ such that $s_\e\to \infty$ as $\e\to 0$, limits
\eqref{e:limsup} and \eqref{e:liminf} given in Lemma~\ref{lem:inecut} yield that
\begin{equation}\label{e:limsup56}
\limsup_{\e\to 0}\frac{\Wp(X^\e_{s_\e}(z),\mu^\e)}{\e^{1\wedge p}}=
\limsup_{\e\to 0}\Wp\Big(\frac{\Sg(s_\e)z}{\e}+\widetilde{L}_\infty, \widetilde{L}_\infty\Big)
\end{equation}
and
\begin{equation}\label{e:liminf56}
\liminf_{\e\to 0}\frac{\Wp(X^\e_{s_\e}(z),\mu^\e)}{\e^{1\wedge p}}=
\liminf_{\e\to 0}\Wp\Big(\frac{\Sg(s_\e)z}{\e}+\widetilde{L}_\infty, \widetilde{L}_\infty\Big).
\end{equation}
}
The choice of $t_\e$ given in \eqref{eq:timecomplejo} and triangle inequality imply 
\begin{align*}
\left|\frac{\Sg(t_\e+\vr)z}{\e}\right|&=e^{-(\gamma/2) \vr }\left|e^{(\gamma/2) (t_{\e}+\vr)}\Sg(t_\e+\vr)z\right|\\
&\lqq e^{-(\gamma/2) \vr }\left|e^{(\gamma/2) (t_{\e}+\vr)}\Sg(t_\e+\vr)z-v(t_{\e}+\vr,z)\right|+e^{-(\gamma/2) \vr}\left|v(t_{\e}+\vr,z)\right|.
\end{align*}
Conversely,
\begin{align*}
e^{-(\gamma/2) \vr}\left|v(t_{\e}+\vr,z)\right|&\lqq 
e^{-(\gamma/2) \vr }\left|e^{(\gamma/2) (t_{\e}+\vr)}\Sg(t_\e+\vr)z-v(t_{\e}+\vr,z)\right|+
\left|\frac{\Sg(t_\e+\vr)z}{\e}\right|.
\end{align*}
The preceding inequalities and the subadditivity of the map $[0,\infty)\ni r \to r^{1\wedge p}$ 
with the help of Proposition~\ref{prop:casos} item II) imply
\begin{equation}\label{eq:arabcota}
\begin{split} 
\limsup\limits_{\e\to 0}\left|\frac{\Sg(t_\e+\vr)z}{\e}\right|^{1\wedge p}
&\lqq 
\limsup\limits_{\e\to 0}
e^{-(\gamma/2) \vr(1\wedge p)}\left|v(t_{\e}+\vr,z)\right|^{1\wedge p}\lqq e^{-\gamma/2 \vr({1\wedge p})}
\limsup\limits_{t\to \infty}
\left|v(t,z)\right|^{1\wedge p},
 \\
\liminf\limits_{\e\to 0}
\left|\frac{\Sg(t_\e+\vr)z}{\e}\right|^{1\wedge p}
&\gqq 
\liminf\limits_{\e\to 0}
e^{-(\gamma/2) \vr(1\wedge p)}\left|v(t_{\e}+\vr,z)\right|^{1\wedge p}\gqq e^{-(\gamma/2) \vr (1\wedge p)}
\liminf\limits_{t\to \infty}
\left|v(t,z)\right|^{1\wedge p}.
\end{split}
\end{equation}
Moreover,
\begin{equation}\label{eq:beab2}
0<
\liminf\limits_{t\to \infty}
\left|v(t,z)\right|^{1\wedge p}\lqq \limsup\limits_{t\to \infty}
\left|v(t,z)\right|^{1\wedge p}<\infty.
\end{equation}
Hence, \eqref{eq:cotasinfsup} in item II) of Proposition~\ref{prop:casos} yields
\begin{align}\label{eq:0infty} 
\lim\limits_{\vr \to \infty}\limsup\limits_{\e\to 0}\left|\frac{\Sg(t_\e+\vr)z}{\e}\right|^{1\wedge p}=0
\quad \textrm{ and }\quad
\lim\limits_{\vr \to -\infty}\liminf\limits_{\e\to 0}
\left|\frac{\Sg(t_\e+\vr)z}{\e}\right|^{1\wedge p}
=\infty.
\end{align}
On the one hand, by \eqref{e:cotauni} given in  item (d) of Lemma~\ref{lem:properties} and \eqref{eq:0infty} we have
\[
\lim\limits_{\vr \to \infty}\limsup_{\e\to 0}\Wp\Big(\frac{\Sg(t_\e+\vr)z}{\e}+\widetilde{\mathcal{L}}_\infty, \widetilde{\mathcal{L}}_\infty\Big)\lqq 
\lim\limits_{\vr \to \infty}\limsup_{\e\to 0}\left|\frac{\Sg(t_\e+\vr)z}{\e}\right|^{1\wedge p}
=0,
\]
which implies the second limit in \eqref{eq:inftylimitzero}.
On the other hand,
by \eqref{e:cotauni} given in  item (d) of Lemma~\ref{lem:properties} and \eqref{eq:0infty} we have 
\begin{equation}
\begin{split}
\liminf_{\e\to 0}\Wp\Big(\frac{\Sg(t_\e+\vr)z}{\e}+\widetilde{\mathcal{L}}_\infty, \widetilde{\mathcal{L}}_\infty\Big)\gqq 
\liminf_{\e\to 0}\left|\frac{\Sg(t_\e+\vr)z}{\e}\right|^{1\wedge p}-2\EE\left[\left|\widetilde{\mathcal{L}}_\infty\right|^{1\wedge p}\right].
\end{split}
\end{equation}
Sending $\vr\to -\infty$ in both sides of the preceding inequality,
\eqref{e:liminf56} and \eqref{eq:0infty}
imply the first limit of \eqref{eq:inftylimitzero}. 
\end{proof}
\noindent
In case of overdamping the cutoff profile can be calculated explicitly. 
\begin{corollary}[Wave cutoff profile for $p\gqq 1$, in case of overdamping $(\gamma /2)^2 > \la_1$]\label{cor:profilewave1}
Assume Hypotheses~\ref{h:sadjwave} and \ref{h:moment} for some $p\gqq 1$.
We keep the notation of Proposition~\ref{prop:casos} and consider case I).
Then for all $z\in \HL$, $z\neq 0$, $\omega_*:=\omega_*(z)$, $v:=v(z)$ 
and $t_\e$ being as in \eqref{eq:timereal}
we have the following profile cutoff phenomenon with 
the explicit cutoff profile 
\[
\lim\limits_{\e\to 0}
\frac{\Wp(X^\e_{t_\e+\vr}(z),\mu^\e)}{\e}= e^{-\omega_* \vr }|v(z)|=:\mathcal{P}_{z}(\vr)\quad \textrm{ for any }\quad \vr\in \mathbb{R}.
\]
\end{corollary}
\begin{proof}
The proof is a direct consequence of item (d) in Lemma~\ref{lem:properties} applied to the right-hand side of \eqref{eq:hjk} in Theorem~\ref{th:abstwave}.
\end{proof}
\noindent
In case of profile cutoff the error can be estimated explicitly. 
\begin{corollary}[Generic profile error estimate]\label{cor:profilewave2}
Let the assumptions of Corollary~\ref{cor:profilewave1} be satisfied and $z\in \HL$, $z\neq 0$.  
Then for all $\vr\in \RR$ and $\e\in (0,1)$ such that $t_\e + \vr\gqq 0$, we have 
\begin{align}
\left|
\frac{\Wp(X^\e_{t_\e+\vr}(z),\mu^\e)}{\e}-\pP_{z}(\vr)\right|\lqq 
\e^{\nicefrac{\la_*}{\omega_*}}\cdot C_*~\mathbb{E}[|\widetilde{\mathcal{L}}_\infty|]~e^{-\la_*\vr}+
\e^{\nicefrac{\beta}{\omega_*}}\cdot C_1(z)~e^{(-\omega_*+\beta) \vr},
\end{align}
where $C_*, \lambda_*$ are given in \eqref{eq:estrella} and $\beta, C_1(z)$ are given in \eqref{eq:beta1}, \eqref{eq:beta2} and \eqref{eq:C1(z)} below.
\end{corollary}
\begin{proof}
The (standard and inverse) triangle inequality with the help of \eqref{e:error1} in Corollary~\ref{cor:inecut090}, \eqref{e:desigual090}, Lemma~\ref{l:convWp} and \eqref{eq:estrella} imply
\begin{align*}
\left|
\frac{\Wp(X^\e_{t_\e+\vr}(z),\mu^\e)}{\e}-\pP_{z}(\vr)\right|
&\lqq 
\left|
\frac{\Wp(X^\e_{t_\e+\vr}(z),\mu^\e)}{\e}-\left|\frac{\Sg(t_\e+\vr)z}{\e}\right|\right|+\left|\frac{\Sg(t_\e+\vr)z}{\e}- e^{-\omega_* \vr }v(z)\right|\\
&\lqq \Wp(\widetilde{\mathcal{L}}_{t_\e+\vr}, \widetilde{\mathcal{L}}_\infty)+e^{-\vr\omega_*}\left|e^{\omega_* (t_\e+\vr)}\Sg(t_\e+\vr)z-v(z) \right|\\
&\lqq 
C_* e^{-\la_* (t_\e+\vr)}\mathbb{E}[|\widetilde{\mathcal{L}}_\infty|]+
e^{-\vr\omega_*} e^{\beta (t_\e+\vr)} C_1(z)\\
&= C_* \e^{\nicefrac{\la_*}{\omega_*}}\cdot e^{-\la_*\vr}\mathbb{E}[|\widetilde{\mathcal{L}}_\infty|]+
\e^{\nicefrac{\beta}{\omega_*}}\cdot e^{(-\omega_*+\beta) \vr} C_1(z).
\end{align*}
\end{proof}
\noindent Profile and window cutoff both imply the simple cutoff phenomenon. 
\begin{corollary}[Simple cutoff phenomenon]\label{cor:profilewave3}
Assume the hypotheses, notation of Proposition~\ref{prop:casos}. 
Then we have the following
\[
\lim\limits_{\e\to 0}
\frac{\Wp(X^\e_{\delta \cdot t_\e}(z),\mu^\e)}{\e^{1\wedge p}}=
\begin{cases}
\infty & \textrm{ for } \delta\in (0,1),\\
0 & \textrm{ for } \delta\in (1,\infty).
\end{cases}
\]
\end{corollary}
\noindent
The proof is virtually identical to the proof of Corollary~3.2 in~\cite{BHPWA} with $e^{-\mathcal{Q}t}$ being replaced by $\Sg(t)$.
Analogously to Corollary~\ref{cor:largedataheat}, we have the following cutoff phenomenon for large initial data.
\begin{corollary}[Large initial data for the wave equation]
Assume the general hypotheses and the notation of Theorem~\ref{th:abstwave} for some $p>0$. 
\begin{enumerate}
\item[a)] Under the specific assumptions of item i) of Theorem~\ref{th:abstwave} for some $z\in \mathcal{H}_\Lambda$, $z\neq 0$, it follows that
\begin{equation*}
\lim\limits_{\e\to 0}
{\Wp(X^1_{t_\e+\vr}(z/\e),\mu^1)}=\mathcal{P}_{p,z}(\vr)\quad 
\textrm{ for any } \quad \vr\in \mathbb{R}.
\end{equation*}
\item[b)] 
Under the specific assumptions of item ii) of Theorem~\ref{th:abstwave} for some $z\in \mathcal{H}_\Lambda$, $z\neq 0$, it follows that
\begin{equation*}
\lim\limits_{\vr\to -\infty}\liminf\limits_{\e\to 0}
{\Wp(X^1_{t_\e+\vr}(z/\e),\mu^1)}=\infty
\quad \textrm{ and } \quad
\lim\limits_{\vr\to \infty}\limsup\limits_{\e\to 0}
{\Wp(X^1_{t_\e+\vr}(z/\e),\mu^1)}=0.
\end{equation*}
\end{enumerate}
The analogues of Corollary~\ref{cor:profilewave1},  Corollary~\ref{cor:profilewave2} and Corollary~\ref{cor:profilewave3} remain valid when 
$\frac{\Wp(X^\e_{\cdot}(z),\mu^\e)}{{\e^{1\wedge p}}}$ is replaced by $\Wp(X^{1}_{\cdot}(z/\e),\mu^1)$.
\end{corollary}
\begin{proof}
It follows directly from Corollary~\ref{cor:largedata}.
\end{proof}
\noindent
The following two lemmas give the proof of Proposition~\ref{prop:casos}.
\begin{lemma}[Generic overdamping, Case I) of Proposition~\ref{prop:casos}]\label{lem:spectral}
Assume Hypotheses~\ref{h:sadjwave} and \ref{h:nonresonance}. 
Moreover, let $\gamma^2>4\la_1$.
Then for any $z\in \mathcal{H}_{\Lambda}$, $z\neq 0$, there exists $\omega_*>0$ and $v=v(z)\in \mathcal{H}$, $v\neq 0$,  defined in \eqref{eq:omegaestrella1} and \eqref{eq:omegaestrella2} such that
\[
\lim\limits_{t\to \infty}
e^{\omega_* t}\Sg(t)z=v(z).
\]
\end{lemma}
\begin{proof}[Proof of Lemma~\ref{lem:spectral}]
For  any $z\in \HL$ write
\[
z=\sum\limits_{j=1}^{j_*}( a_{j}\hat v_{j}+ a_{-j}\hat v_{-j})+\sum_{j=j_*+1}^{\infty}(b_{j}v_{j}+b_{-j}v_{-j}),
\]
where $j_*$ is given in  
\eqref{e:jotaestrella}.
Assume (otherwise  Case II) in Lemma~\ref{lem:spectral2} is satisfied) 
\begin{equation}\label{eq:sumanozero}
\sum\limits_{j=1}^{j_*}(|a_{j}|+| a_{-j}|)>0.
\end{equation}
Since \eqref{eq:sumanozero},
we define
\begin{align}
j_0&:=\max\{j\in\{-j_*,\ldots,j_*\}\setminus\{0\}:| a_j|>0\}.
\end{align}
For any $t\gqq 0$, we note that
\[
\Sg(t)z=\sum\limits_{j=1}^{j_0}(e^{\hat \omega_{j}t} a_{j}\hat v_{1}+e^{\hat \omega_{-j}t} a_{-j}\hat v_{-j})+\sum_{j=j_*+1}^{\infty}(e^{\omega_{j}t}b_{j}v_{j}+e^{\omega_{-j}t}b_{-j}v_{-j}).
\]
Note that $0>\hat\omega_j>\hat\omega_{j+1}$ and $\hat \omega_{-j}<\hat\omega_{-(j+1)}<0$ 
for $j\in \{{1,\ldots,j_0-1\}}$. Also, $\hat\omega_j>\hat\omega_{-k}$ for any $j,k\in \{{1,\ldots,j_0-1\}}$. We distinguish the following two cases.
\begin{enumerate}
 \item[a)] In case of $J:=\{j\in \{1,\ldots ,j_1\}:a_{j}=0\}\neq \emptyset$, let $j_1:=\min(J)$, 
 \begin{equation}\label{eq:omegaestrella1}
\omega_*=-\hat{\omega}_{j_1}>0 \qquad \mbox{ and } \qquad v(z) =  a_{j_1} \hat v_{j_1}.
\end{equation}
Hence 
\begin{equation}\label{eq:form91}
\lim\limits_{t\to \infty}\exp\left(t\omega_*\right)\Sg(t)z=v(z)\neq 0.
\end{equation}
\item[b)] In case of $J=\emptyset$, let
$j_2=\max\{j\in \{1,\ldots j_1\}:a_{-j}\neq 0\}$,
 \begin{equation}\label{eq:omegaestrella2}
\omega_*=-\hat \omega_{-j_2}>0\qquad \mbox{ and } \qquad v(z) =  a_{-j_2} \hat v_{-j_2}.
\end{equation}
Hence 
\begin{equation}\label{eq:form29}
\lim\limits_{t\to \infty}\exp\left(t\omega_*\right)\Sg(t)z=v(z)\neq 0.
\end{equation}
\end{enumerate}
In Case a) we denote 
\begin{align*}
&\delta := -\frac{\sqrt{\gamma^2-4\la_{j_1}}}{2} = \omega_*+\mathsf{Re}(\omega_{\pm j}) \quad \mbox{ for all }\quad j\gqq j_*+1,  \\
&\delta_{-}:=\max\limits_{1\lqq j\lqq j_0}\{\omega_*+\hat{\omega}_{-j}\}<0, \qquad \mbox{ and }\quad \delta_{+}:=\max\limits_{j_1+1\lqq j\lqq j_0}\{\omega_*+\hat{\omega}_{j}\}<0.
\end{align*}
Then for 
\begin{equation}\label{eq:beta1} \beta:=\max\{\delta_{-},\delta_{+},\delta\}<0 \end{equation} we have 
\begin{align}
\left|\exp\left(t\omega_*\right)\Sg(t)z- a_{j_1} \hat v_{j_1}\right|
&\lqq 
\sum\limits_{j=1}^{j_0}(e^{\delta_{+}t}|a_{j}||\hat v_{1}|+e^{ \delta_{-}t}| a_{-j}||\hat v_{-j}|) +e^{\delta t}
\sum_{j=j_*+1}^{\infty}(|b_{j}||v_{j}|+|b_{-j}||v_{-j}|)\lqq C_1(z) e^{\beta t},\label{e:errorest}
\end{align}
where 
\begin{equation}\label{eq:C1(z)}
C_1(z) := \left(\sum\limits_{j=1}^{j_0}| a_{j}||\hat v_{1}|+| a_{-j}||\hat v_{-j}|)+\sum_{j=j_*+1}^{\infty}(|b_{j}||v_{j}|+|b_{-j}||v_{-j}|)\right).
\end{equation}
In Case b) we obtain for 
\begin{align}
&\delta := -\frac{\sqrt{\gamma^2-4\la_{j_2}}}{2} = \omega_*+\mathsf{Re}(\omega_{\pm j}) \quad \mbox{ for all }\quad j\gqq j_*+1,\nonumber\\
&\delta_{-}:=\max\limits_{j_2+1\lqq j\lqq j_0}\{\omega_*+\hat{\omega}_{-j}\}<0,\qquad 
\mbox{ and }\quad \beta:=\max\{\delta_{-},\delta_{+},\delta\}<0\label{eq:beta2}
\end{align}
the estimate 
\begin{align}
\left|\exp\left(t\omega_*\right)\Sg(t)z- a_{-j_2} \hat v_{-j_2}\right|
&\lqq C_1(z) e^{\beta t}.
\label{e:errorest1}
\end{align}
\end{proof}
\noindent
The case of subcritical damping $\gamma^2<4\la_1$ treats the situation $j_*=0$ given in \eqref{e:jotaestrella}.

\begin{lemma}[Generic subcritical damping, Case II) of Proposition~\ref{prop:casos}]\label{lem:spectral2}
Assume Hypotheses~\ref{h:sadjwave} and \ref{h:nonresonance}.
Moreover, let $\gamma^2<4\la_1$. 
Then for any $z\in \mathcal{H}_{\Lambda}$, $z\neq 0$, there exists $v=v(t, z)\in \mathcal{H}$ such that
\[
\lim\limits_{t\to \infty}
|e^{(\gamma/2) t}\Sg(t)z-v(t,z)|=0,
\]
where $v(t,z)$ is given in \eqref{eq:v(t,z)}. In addition, we have 
\begin{equation}\label{eq:infsup}
0<\liminf\limits_{t\to \infty}|v(t,z)|\lqq 
\limsup\limits_{t\to \infty}|v(t,z)|<\infty.
\end{equation}
\end{lemma}

\begin{proof}[Proof of Lemma~\ref{lem:spectral2}] 
For any $z\in \HL$ we write
\begin{align*}
z&=\sum_{j\in \NN} b_j v_{j}
+\sum_{j\in \NN}b_{-j} v_{-j}=\sum_{j\in \NN} b_j v_{j}
+\sum_{j\in \NN} b_{-j} \bar v_{j}.
\end{align*}
Hence, for any $t\gqq 0$ we have
\begin{align}\label{eq:v(t,z)}
e^{t(\gamma/2)}\Sg(t)z&=\sum_{j\in \NN}e^{\ii \theta_j t}  b_j v_{j}
+\sum_{j\in \NN}e^{-\ii\theta_{j} t}b_{-j} v_{-j}=\sum_{j\in \NN}e^{\ii \theta_j t} b_j v_{j}
+\sum_{j\in \NN}e^{-\ii\theta_{j} t} b_{-j} \bar v_{j}=: v(t, z),
\end{align}
where $\theta_{\pm j}$ are the arguments of $\omega_{\pm _j}$.
A straightforward calculation yields
\begin{align*}
\left|e^{t(\gamma/2)}\Sg(t)z\right|^2
&=\dv\sum_{j\in \NN}e^{\ii \theta_j t} b_j v_{j}
+\sum_{j\in \NN}e^{-\ii\theta_{j} t} b_{-j} \bar v_{j},
\sum_{k\in \NN}e^{\ii \theta_k t} b_k v_{k}
+\sum_{k\in \NN}e^{-\ii\theta_{k} t} b_{-k} \bar v_{k}\dw\\
&\hspace{-2.2cm}=\dv\sum_{j\in \NN}e^{\ii \theta_j t} b_j v_{j},
\sum_{k\in \NN}e^{\ii \theta_k t} b_k v_{k}
\dw+\dv\sum_{j\in \NN}e^{-\ii \theta_j t} b_{-j} \bar v_{j},
\sum_{k\in \NN}e^{-\ii \theta_k t} b_{-k} \bar v_{k}\dw +2\Rea \Big(\dv \sum_{j\in \NN}e^{\ii\theta_{j} t} b_j v_{j},\sum_{k\in \NN}e^{-\ii \theta_k t} {b}_{-k} \bar{v}_{k} \dw\Big).
\end{align*}
Then the explicit computations of the inner products \eqref{eq:innerproducts} yield for any $z\in \HL$ 
\begin{equation}\label{eq:modulo}
\begin{split}
|e^{t(\gamma/2)}\Sg(t)z|^2=&\sum_{j\in \NN} |b_j|^2(1+\la_{j}+|\omega_{j}|^2) +\sum_{j\in \NN}|b_{-j}|^2(1+\la_{j}+|\omega_{j}|^2)+2\Rea \Big(\sum_{j\in \NN}e^{2\ii\theta_{j} t} b_j\bar b_{-j} (1+\la_j+\omega^2_j)\Big).
\end{split}
\end{equation}
We observe that $|\omega_j|^2=\la_j$ for all $j\in \NN$. By Young's inequality we have 
\begin{align}
&\label{eq:modulo0}\\
&|e^{2\ii\theta_{j} t} b_j\bar b_{-j} (1+\la_j+\omega^2_j)|&\lqq |b_j||b_{-j}|(1+\la_j+|\omega_j|^2)
\lqq |b_j||b_{-j}||(1+2\la_j)\lqq \frac{|b_j|^2(1+2\la_j)}{2}+\frac{|b_{-j}|^2(1+2\la_j)}{2}.\nonumber
\end{align}
Since $z\in \HL$, the right-hand side \eqref{eq:modulo0} is sumable over all $j\in \NN$ and hence 
the right-hand side of \eqref{eq:modulo} is finite. In addition,
\begin{align}\label{eq:limsonda}
\limsup\limits_{t\to \infty}|e^{t(\gamma/2)}\Sg(t)z|^2\lqq 3\sum_{j\gqq 1} |b_j|^2(1+\la_j) +3\sum_{j\gqq 1}|b_{-j}|^2(1+\la_j)<\infty.
\end{align}
In the sequel, we prove that 
\begin{align}\label{eq:limfonda}
\liminf\limits_{t\to \infty}|e^{t\gamma/2}\Sg(t)z|^2>0.
\end{align}
Since $z\neq 0$, there exists $j_0\in \mathbb{N}$ such that $|b_{j_0}|^2+|b_{-j_0}|^2>0$. By \eqref{eq:modulo} we have
\begin{equation}\label{eq:normbound}
\begin{split}
|e^{t\gamma/2}\Sg(t)z|^2&=
(1+\la_{j_0}+|\omega_{j_0}|^2)(|b_{j_0}|^2+|b_{-j_0}|^2)+
2\Rea\Big((e^{2\ii\theta_{j_0} t} b_{j_0}\bar b_{-j_0})(1+\la_{j_0}+\omega^2_{j_0})\Big)\\[3mm]
&\qquad+
\sum_{
\substack{j\in \NN\\
j\neq j_0}
} |b_j|^2(1+\la_{j}+|\omega_{j}|^2) +\sum_{\substack{j\in \NN\\
j\neq j_0}}|b_{-j}|^2(1+\la_{j}+|\omega_{j}|^2)+2\Rea \Big(\sum_{\substack{j\in \NN\\
j\neq j_0}}e^{2\ii\theta_{j} t} b_j\bar b_{-j} (1+\la_{j}+\omega^2_{j})\Big)\\
&=|e^{\ii\theta_{j_0}t}b_{j_0}v_{j_0}+e^{-\ii\theta_{j_0}t}b_{-j_0}\bar v_{j_0}|^2+
\Big|\sum_{\substack{j\in \NN\\
j\neq j_0}}
e^{\ii\theta_{j}t}b_{j}v_{j}+e^{-\ii\theta_{j}t}b_{-j}\bar v_{j} \Big|^2\\
&\gqq |e^{\ii\theta_{j_0}t}b_{j_0}v_{j_0}+e^{-\ii\theta_{j_0}t}b_{-j_0}\bar v_{j_0}|^2.
\end{split}
\end{equation}
We claim that
\begin{equation}\label{eq:claimj0}
\liminf\limits_{t\to \infty}
|e^{\ii\theta_{j_0}t}b_{j_0}v_{j_0}+e^{-\ii\theta_{j_0}t}b_{-j_0}\bar v_{j_0}|^2>0.
\end{equation}
Indeed, assume by contradiction that there exists a sequence $(t_k)_{k\in \mathbb{N}}$, $t_k\to \infty$ as $k\to \infty$ satisfying
\begin{equation}\label{eq:limj}
\lim_{k\to \infty}
|e^{\ii\theta_{j_0}t_k}b_{j_0}v_{j_0}+e^{-\ii\theta_{j_0}t_k}b_{-j_0}\bar v_{j_0}|^2=0.
\end{equation}
By the Bolzano-Weierstrass theorem we obtain the existence of a subsequence $(t_{k_{m}})_{m\in \mathbb{N}}$
of $(t_k)_{k\in \mathbb{N}}$ such that 
\[
\lim\limits_{m\to  \infty}e^{ \ii\theta_{j_0}t_{k_m}}=z^{+}\quad\textrm{ and }\quad
\lim\limits_{m\to  \infty}e^{- \ii\theta_{j_0}t_{k_m}}=z^{-}.
\]
By \eqref{eq:limj} we deduce 
\[
z^+ b_{j_0}v_{j_0}+z^{-}b_{-j_0}\bar v_{j_0}=0.
\]
Since $v_{j_0}$ and $\bar v_{j_0}$ are linearly independent and $|z^{-}|=|z^{+}|=1$, we obtain $b_{j_0}=b_{-j_0}=0$ which yields a contradiction to the choice of $j_0$. 
Combining \eqref{eq:normbound} and \eqref{eq:claimj0} we infer \eqref{eq:limfonda}.
\end{proof}

\bigskip 
\section{Profile cutoff for the heat equation with multiplicative noise}\label{s:HM}
\noindent
This section treats the stochastic heat equation with multiplicative noise of the following type
\begin{equation}
\left\{
\begin{array}{r@{\;=\;}l}
\ud X^\e_t &A X^\e_t \ud t+  
\e X^\e_t \ud L_t \quad  
\textrm{  for any }\quad t\gqq 0,\\
X^\e_0 & h\in H
\end{array}
    \right.
\end{equation}
where $A$ satisfies Hypothesis~\ref{h:sadj}. We consider the space $L_2(H,H)$ equipped with the Hilbert-Schmidt norm $\|\cdot \|$.
\bigskip 
\subsection{Multiplicative $Q$-Brownian motion}\label{ss:HMB}
Let Hypothesis~\ref{h:sadj} be satisfied for $A$ 
and $(e_j)_{j\in \mathbb{N}_0}$ be the orthonormal basis of eigenvectors of $A$ in $L_2(H,H)$. We consider the following diagonal structures for the linear operators 
\begin{equation}\label{eq:comutatividad}
G_i=\textsf{diag}(g^i_j)_{j\in \mathbb{N}},  \quad i=1,2\quad \textrm{ with respect to the basis } \quad (e_j)_{j\in \mathbb{N}}.
\end{equation}
We study the solution semiflow of the heat equation with multiplicative noise  given by
\begin{equation}\label{e:multi}
\left\{
\begin{array}{r@{\;=\;}l}
\ud \xX^\e_t &A \xX^\e_t \ud t+  
\e G_1 \xX^\e_t \ud B^{(1)}_t+\e G_2 \xX^\e_t \ud B^{(2)}_t \quad  
\textrm{  for any }\quad t\gqq 0,\\
\xX^\e_0 & h\in H,
\end{array}
    \right.
\end{equation}
where $\e\in (0,1)$, with the independence condition:
\begin{align}
B^{(1)}, B^{(2)} \quad \textrm{are mutually independent one dimensional standard Brownian motions.}\label{eq:independencia}
\end{align}
We point out that the diagonal structure \eqref{eq:comutatividad} implies
\begin{equation*}
G_iG^*_j=G^*_jG_i,\quad
G_iG_j=G_jG_i\quad \textrm{ and }\quad
G_j A = A G_j\quad \textrm{ for any }\quad i,j\in \{1,2\}.
\end{equation*}
The existence and uniqueness of \eqref{e:multi}  is given in 1.2(a) of \cite{Salaheldin}.
We stress that
the commutative relations are natural in the case of the stochastic heat equation, whereas,
in the case of the stochastic wave equation, the commutative of the respective operators $A$ and $G_j$ is too restrictive in order to cover the natural case of noise acting only in the velocity component.

\noindent
By Theorem~16.5 in~\cite{PZ07}, p.290, we have the existence of a unique invariant probability measure $\mu^\e$ for \eqref{e:multi}. It is not hard to see that $\mu^\e=\delta_{0}$.
Then
\begin{equation}\label{eq:distancia}
\mathcal{W}^2_{2}(\xX^\e_t(h),\mu^\e)=\mathbb{E}[|\xX^\e_t(h)|^2].
\end{equation}
For $G_2=0$, formula~(1) p.~106 in \cite{DapratoIannelli}  states
\begin{equation*}
\xX^\e_t(h)=\exp\left(\left(A-\frac{1}{2}\e^2 G^2_1\right)t\right)\exp\left(\e G_1B^{(1)}_t\right)h.
\end{equation*}
{
Since $h=\sum_{j\in \mathbb{N}_0} \<h,e_j\>e_j$,
the diagonal structure \eqref{eq:comutatividad}  implies
\begin{equation*}
\xX^\e_t(h)=\exp\left(\left(A-\frac{1}{2}\e^2 G^2_1\right)t+\e G_1B^{(1)}_t\right)h.
\end{equation*}
}
Following step by step the proof of formula~(1) in \cite{DapratoIannelli} under \eqref{eq:comutatividad} and \eqref{eq:independencia}, it is not hard to see that
\begin{equation}\label{eq:formulilla0}
\xX^\e_t(h)=\exp\left(\left(A-\frac{1}{2}\e^2 G^2_1-\frac{1}{2}\e^2 G^2_2\right)t\right)\exp\left(\e G_1B^{(1)}_t+\e G_2B^{(2)}_t\right)h.
\end{equation}
The finite dimensional case of 
this representation has been studied in formula (4.11) p. 100 in \cite{Mao2008}.
{
We stress that by \eqref{eq:comutatividad}
\[
\exp\left(\left(A-\frac{1}{2}\e^2 G^2_1-\frac{1}{2}\e^2 G^2_2\right)t\right)
\quad \textrm{ and } \quad
\exp\left(\e G_1B^{(1)}_t+\e G_2B^{(2)}_t\right)
\]
are both diagonal operators with respect to $(e_j)_{j\in \mathbb{N}_0}$
and hence commute.
Consequently,
\begin{equation}\label{eq:formulilla}
\xX^\e_t(h)=\exp\left(\aA_{\e} t+\e G_1B^{(1)}_t+\e G_2B^{(2)}_t\right)h,\quad \textrm{ 
where }\quad
\mathcal{A}_\e:=A-\frac{1}{2}\e^2 G^2_1-\frac{1}{2}\e^2 G^2_2.
\end{equation}
}
Since $A$ satisfies Hypothesis~\ref{h:sadj}, Lemma~\ref{l:jaraheat} implies the following.
For each $h\in H$, $h\neq 0$,
there exist an index $N_h\in \NN_0$  and $ v_h\in H$
such that $ v_h\neq 0$ and
\begin{equation}\label{e:jaraheatm09}
\lim\limits_{t\to \infty}|e^{t  \la_{N_h}}e^{t A} h- v_h|=0.
\end{equation}
Recall  that $\mu^\e=\delta_{0}$ for any $\e\in (0,1]$. In this setting, we obtain the profile cutoff phenomenon in the following sense. Note that the statement is given in the Wasserstein distance of order 2 due to the particular bilinear structure of \eqref{eq:distancia} which allows to use the adjoint operators and express the right-hand side of \eqref{eq:distancia} as a square.
\begin{theorem}[Profile cutoff for the stochastic heat equation with multiplicative noise]\label{t:multprofile}
Assume $A$ satisfies Hypothesis~\ref{h:sadj} and let $h\in H$, $h\neq 0$.
We also assume \eqref{eq:comutatividad} and \eqref{eq:independencia} hold true.
Let $(a_\e)_{\e\in (0,1)}$ be a positive  function such that 
\begin{equation}\label{eq:doblelimite}
\lim\limits_{\e\to 0}a_\e=\lim\limits_{\e\to 0} \e^2|\ln(a_\e)|=0.
\end{equation}
We keep the notation of  \eqref{e:jaraheatm09} and define
\begin{equation}\label{e:tildetepsilon}
 t_\e := \frac{1}{\la_{N_h}}|\ln(a_\e)|.  
\end{equation}
Then for any $\vr\in \RR$ we have the profile cutoff phenomenon 
\begin{align*}
\lim_{\e\ra 0}  \frac{\mathcal{W}_{2}(\xX^\e_{ t_\e + \vr} (h),\mu^\e)}{a_\e} = e^{-\la_{N_h} \vr} |v_h|=:\mathcal{P}_h(\vr).
\end{align*}
\end{theorem}
\begin{proof}
We start with the computation of the right-hand side of \eqref{eq:distancia}. Recall the representation \eqref{eq:formulilla}.
Then  for any  $h\in H$, $\e\in (0,1)$ and $t\gqq 0$ we have
\begin{align*}
\mathbb{E}\left[|\xX^\e_t(h)|^2\right]
&=
\mathbb{E}\left[\left\<\exp\left(\aA_{\e} t+\e G_1B^{(1)}_t+\e G_2B^{(2)}_t\right)h,\exp\left(\aA_{\e} t+\e G_1B^{(1)}_t+\e G_2B^{(2)}_t\right)h\right\>\right]\\
&=\mathbb{E}\left[\left\<h,\exp\left(\aA^*_{\e} t+\e G^*_1B^{(1)}_t+\e G^*_2B^{(2)}_t\right)\exp\left(\aA_{\e} t+\e G_1B^{(1)}_t+\e G_2B^{(2)}_t\right)
h\right\>\right]\\
&=\left\<h,\exp(\aA^*_{\e} t)\mathbb{E}\left[\exp\left(\e(G_1+G^*_1)B^{(1)}_t\right)\right]
\mathbb{E}\left[\exp\left(\e(G_2+G^*_2)B^{(2)}_t\right)\right]\exp(\aA_{\e} t) h\right\>,
\end{align*}
where in the last equality we have used \eqref{eq:comutatividad} and \eqref{eq:independencia}.
Since $G_1+G^*_1$ and $G_2+G^*_2$ are  symmetric operators, a standard diagonalization argument yields
\[
\mathbb{E}\left[\exp\left(\e(G_i+G^*_i)B^{(i)}_t\right)\right]=\exp\left(\frac{1}{2}\e^2(G_i+G^*_i)^2 t\right) \quad \textrm{ for }\quad i\in \{1,2\}.
\]
Hence
\begin{align*}
\mathbb{E}\left[|\xX^\e_t(h)|^2\right]
&=\left\<h,\exp(\aA^*_{\e} t)\exp\left(\frac{1}{2}\e^2(G_1+G^*_1)^2 t\right)
\exp\left(\frac{1}{2}\e^2(G_2+G^*_2)^2 t\right)\exp(\aA_{\e} t)h\right\>\\
&=\left\<h,\exp(A t-\frac{t}{2}\e^2 (G^*_1)^2-\frac{t}{2}\e^2 (G^*_2)^2)\exp\left(\frac{1}{2}\e^2 (G_1+G^*_1)^2 t\right)\right.\\
&\quad\quad \cdot \left.
\exp\left(\frac{1}{2}\e^2(G_2+G^*_2)^2 t\right)\exp\left(A t-\frac{t}{2}\e^2 G_1^2-\frac{t}{2} \e^2 G_2^2\right)h\right\>\\
&=\left\<h,\exp(A t)\exp(\e^2 G_1G^*_1 t+\e^2 G_2G^*_2 t)\exp(A t)h\right\>\\
&=\left\<\exp(tA+(t/2)\e^2 G_1G^*_1+(t/2)\e^2 G_2G^*_2)h,\exp(tA+(t/2)\e^2 G_1G^*_1+(t/2)\e^2 G_2G^*_2)h\right\>\\
&=|\exp(t \widetilde \aA_{\e})h|^2,
\end{align*}
where 
$
\widetilde \aA_\e := A+\frac{1}{2} \e^2 G_1 G^*_1+\frac{1}{2}\e^2 G_2G^*_2$. 
Since $\frac{1}{a_\e} = e^{ \la_{N_h}  t_\e}$, we have 
\begin{align*}
\frac{\mathcal{W}_2(\xX^\e_{ t_\e + \vr}(h), \mu^\e)}{a_\e}
&=  \frac{1}{a_\e}|\exp(( t_\e + \vr) \widetilde \aA_{\e})h|=  e^{- \la_{N_h} \vr} |e^{ \la_{N_h} ( t_\e + \vr)} \exp(( t_\e + \vr) \widetilde \aA_{\e})h|\\
&
=e^{- \la_{N_h} \vr} |\exp(( t_\e + \vr) (\frac{\e^2}{2}  G_1 G^*_1+\frac{\e^2}{2} G_2G^*_2))e^{ \la_{N_h} ( t_\e + \vr)} \exp(( t_\e + \vr)  A)h|
 \ra e^{- \la_{N_h} \vr} |  v_h|, 
\end{align*}
as $\e\to 0$,
where in the preceding limit we have used
\eqref{e:jaraheatm09} and limit \eqref{eq:doblelimite}.
\end{proof}
\noindent
Theorem~\ref{t:multprofile} shows profile cutoff for Example~2.1 in~\cite{Ichikawa}
in the case  of small $\e=r$ in his notation.
The following corollary yields the profile cutoff phenomenon for perturbations by a general $Q$-Brownian motion.
\begin{corollary}\label{cor:infinitos}
For a sequence $G_i \in L_2(H, H)$ such that 
\[
\sum_{i\in \NN} \|G_i\|^2 < \infty 
\]
we consider 
\begin{equation}\label{e:multi1}
\left\{
\begin{array}{r@{\;=\;}l}
\ud \xX^\e_t &A \xX^\e_t \ud t+  
\e\sum\limits_{k=1}^\infty
 G_k \xX^\e_t \ud B^{(k)}_t \quad  
\textrm{  for any }\quad t\gqq 0,\\
\xX^\e_0 & h\in H.
\end{array}
\right.
\end{equation}
Let $A$ satisfy Hypothesis~\ref{h:sadj} and let $h\in H$, $h\neq 0$.  
Furthermore, we assume
\begin{equation}\label{eq:comutatividad1}
G_i=\textsf{diag}(g^i_j)_{j\in \mathbb{N}},  \quad i\in \mathbb{N}\quad \textrm{ with respect to the orthonormal basis of eigenvectors $(e_j)_{j\in \mathbb{N}}$ of $A$ in $L_2(H,H)$}
\end{equation} 
and
\begin{align}
\textrm{$\big(B^{(k)}\big)_{k\in \mathbb{N}}$ is an iid family of one dimensional standard Brownian motions.}\label{eq:independencia1}
\end{align}
Then for any $\vr\in \RR$
and $(a_\e)_{\e\in (0,1)}$ being a  positive function satisfying 
\eqref{eq:doblelimite}
 we have the profile cutoff phenomenon 
\begin{align*}
\lim_{\e\ra 0} 
\frac{\mathcal{W}_{2}(\xX^\e_{ t_\e + \vr} (h),\mu^\e)}{a_\e} = e^{- \la_{N_h} \vr} | v_h|=:\mathcal{P}_h(\vr),  
\end{align*}
where $ t_\e$ is defined in \eqref{e:tildetepsilon}, and $\la_{N_h}$ and $v_h$ are given in \eqref{e:jaraheatm09}.
\end{corollary}
\begin{proof}
The proof follows the lines of the proof of Theorem~\ref{t:multprofile} due to 
\[
\mathbb{E}\left[|\xX^\e_t(h)|^2\right]=
\left|\exp\big(t\widetilde \aA_\e\big)h\right|^2,
\quad \textrm{ where }\quad
\widetilde \aA_\e := A+\frac{1}{2}\sum_{i\in \mathbb{N}} \e^2 G_i G^*_i.
\]
\end{proof}

\noindent
Analogously to Corollary~\ref{cor:errorheat}  and Corollary~\ref{cor:simpleheat} we have the following statements.

\begin{corollary}[Profile error estimate]\label{cor:errorheatm45}
Let the assumptions of Corollary~\ref{cor:infinitos} be satisfied.
Then for all $\vr\in \RR$ there exists a positive constant $K:=K(\vr,C_*,\lambda_*)$ such that for all $\e$ small enough with $t_\e + \rho\gqq 0$  we have 
\begin{align}
\left|
\frac{\mathcal{W}_2(\xX^\e_{t_\e+\vr}(h),\mu^\e)}{a_\e}-\mathcal{P}_h(\vr)\right|\lqq K\cdot 
a_\e^{\big(1-\nicefrac{\la_{N_{h}}}{\la_{N^*_{h}}}\big)}|h|
\end{align}
for any $h\in H$, where $C_*$ and $\lambda_*$ are given in \eqref{e:expsta}, and $N_h$, $N^*_{h}$ are defined in \eqref{eq:Nh}.
\end{corollary}
\begin{corollary}[Simple cutoff phenomenon]\label{cor:simpleheatm45}
Let the assumptions of Corollary~\ref{cor:infinitos} be satisfied.
For  $h\in H$, $h\neq 0$
\[
\lim\limits_{\e\to 0}
\frac{\mathcal{W}_2(\xX^\e_{\delta \cdot t_\e}(h),\mu^\e)}{a_\e}=
\begin{cases}
\infty & \textrm{ for } \delta\in (0,1),\\
0 & \textrm{ for } \delta\in (1,\infty),
\end{cases}
\]
where $t_\e$ is given in \eqref{e:tildetepsilon}.
\end{corollary}
\noindent
The linearity of \eqref{e:multi1} in contrast of the affine structure of \eqref{e:OU} changes the reparametrization of the large initial value problem as follows.
\begin{corollary}[Large initial data for the heat equation]\label{cor:heatmuchosm45}
Let the assumptions of Corollary~\ref{cor:infinitos} be satisfied.
For  $h\in H$, $h\neq 0$ it follows that
\begin{equation*}
\lim\limits_{\e \to 0}{\mathcal{W}_2(\xX^{\e}_{t_\e+\vr}(h/a_\e),\mu^\e)}=\pP_{h}(\vr).
\end{equation*}
The analogues of Corollary~\ref{cor:infinitos},  Corollary~\ref{cor:errorheatm45} and Corollary~\ref{cor:simpleheatm45} are valid when 
$\frac{\mathcal{W}_2(\xX^\e_{\cdot}(h),\mu^\e)}{a_\e}$ is replaced by $\mathcal{W}_2(\xX^{\e}_{\cdot}(h/a_\e),\mu^\e)$.
\end{corollary}
\subsection{Multiplicative L\'evy noise}\label{ss:HML}
{In this subsection we consider a L\'evy measure $\nu$
 on the space of diagonal operators in $L_2(H,H)$ with respect to the orthonormal base of eigenvectors $(e_j)_{j\in \mathbb{N}_0}$ of $A$, that is,
\begin{equation}\label{e:comutalev}
\textrm{$\textsf{supp}(\nu)$ is contained in the space of diagonal operators in $L_2(H,H)$ with respect to $(e_j)_{j\in \mathbb{N}_0}$}.
\end{equation}
We then consider a pure jump L\'evy process $(L_t)_{t\gqq 0}$,  $L_t=\int_{0}^t \int_{0< \|z\|<1}z\widetilde{N}(\ud z,\ud t)$,
  where $\widetilde{N}$ is the compensated Poisson random measure associated to $\nu$ on a given probability space $(\Omega, \mathcal{F},\mathbb{P})$ satisfying the usual conditions in the sense of Protter, see \cite{Protter}.
Fix $\eta\in(0,1)$ and let $(\xX^{(\eta, \e)}_t)_{t\gqq 0}$  be the mild solution of the operator valued equation in $L_2(H,H)$ of 
\begin{equation}\label{e:LMN99}
\left\{
\begin{array}{r@{\;=\;}l}
\ud \xX^{(\eta,\e)}_t&A \xX^{(\eta,\e)}_t\ud t +
\e\int_{\eta\lqq \|z\|<1}(z \xX^{(\eta,\e)}_{t-})\, \widetilde{N}(\ud z,\ud t)  \quad  
\textrm{  for any }\quad t\gqq 0,\\
\xX^{(\eta,\e)}_0 & I,
\end{array}
\right.
\end{equation}
where $I$ is the identity on $H$.
That is, it satisfies 
$\mathbb{P}$-almost surely for all $t\gqq 0$
\begin{equation}\label{e:mildsolt}
\begin{split}
\xX^{(\eta,\e)}_t=I+\e\int_{0}^{t} \int_{\eta\lqq \|z\|<1} S(t-s)(z \xX^{(\eta,\e)}_{s-})\, \widetilde{N}(\ud z,\ud s), 
\end{split}
\end{equation}
where $S$ is the semigroup defined in \eqref{e:semi}.
Since $\eta>0$, the existence and uniqueness of \eqref{e:mildsolt} is straightforward by a standard interlacing procedure for compound Poisson processes, see for instance \cite{APPLEBAUMBOOK}. 
}
{
Note that the evaluation $\xX^{(\eta,\e)}_t(h)$ of the unique mild solution of \eqref{e:mildsolt} is well-defined for all $h\in H$. However, a priori we do not have an analogously explicit representation as in \eqref{eq:formulilla}.
For $h\in D(A)$, we obtain the following explicit exponential representation of $(\xX^{(\eta, \e)}_t(h))_{t\gqq 0}$, which coincides with a strong solution of \eqref{e:LMN99}, see the proof of Lemma~\ref{lem:exprepresentlevy} in Appendix~\ref{Ap:A2}.
This representation turns out to be crucial for the cutoff result.}
\begin{lemma}\label{lem:exprepresentlevy}
We assume \eqref{e:comutalev}.
Then we have the following representation
\begin{equation}\label{e:stochexp}
\xX^{(\eta,\e)}_t(h)=\mathcal{E}^{(\eta,\e)}_{t}h,\qquad h \in {D(A)},\quad t\gqq 0,
\end{equation}
where 
$\mathcal{E}^{(\eta,\e)}_{t}:=\exp(\Theta^{(\eta,\e)}_t)$ and 
\begin{align*}
\Theta^{(\eta,\e)}_t&:=
t\Big(A+\e\int_{\eta\lqq \|z\|<1} (\log(I+z)-z) \nu(\ud z) \Big)+\e\int_{0}^{t}\int_{\eta\lqq \|z\|<1}
\log(I+z)\tilde{N}(\ud z, \ud s)
=:t\cdot\mathbb{A}^{(\eta,\e)}+\Lambda^{(\eta,\e)}_t
\end{align*}
{
with the power series 
\begin{equation}\label{eq:formalog}
\log(I+z):=\sum_{j\in \mathbb{N}}(-1)^{j-1}\frac{z^j}{j} \quad \textrm{ for }\|z\|<1.
\end{equation}
}
\end{lemma}
\noindent
The proof is given in Appendix~\ref{Ap:A2}. Note that since the jumps in \eqref{e:LMN99} are uniformly bounded, its solution has finite $p$-th moments for any $p>0$. 
Since $A$ satisfies Hypothesis~\ref{h:sadj}, Lemma~\ref{l:jaraheat} implies the following.
For each $h\in H$, $h\neq 0$, 
there exist an index $N_h\in \NN_0$  and $ v_h\in H$
such that $ v_h\neq 0$ and
\begin{equation}\label{e:jaraheatm09levy}
\lim\limits_{t\to \infty}|e^{t  \la_{N_h}}e^{t A} h- v_h|=0.
\end{equation}
Recall  that $\mu^\e=\delta_{0}$ for any $\e\in (0,1]$. Analogously to Theorem~\ref{t:multprofile} we have the following theorem.
\begin{theorem}[Profile cutoff for the stochastic heat equation with multiplicative L\'evy noise]\label{t:multprofilelevy}
Assume that $A$ satisfies Hypothesis~\ref{h:sadj}, {$h\in D(A)$}, $h\neq 0$, 
and \eqref{e:comutalev}.
Let $(a_\e)_{\e\in (0,1)}$ be a positive  function such that 
\begin{equation}\label{eq:doblelimitelevy}
\lim\limits_{\e\to 0}a_\e=\lim\limits_{\e\to 0} \e|\ln(a_\e)|=0.
\end{equation}
We keep the notation of \eqref{e:jaraheatm09} and define
\begin{equation}\label{e:tildetepsilonlevy}
 t_\e := \frac{1}{\la_{N_h}}|\ln(a_\e)|.  
\end{equation}
Then for any $\vr\in \RR$ we have the profile cutoff phenomenon 
\begin{align*}
\lim_{\e\ra 0}  \frac{\mathcal{W}_{2}(\xX^{(\e,\eta)}_{ t_\e + \vr} (h),\mu^\e)}{a_\e} = e^{-\la_{N_h} \vr} |v_h|=:\mathcal{P}_h(\vr).
\end{align*}
\end{theorem}

\begin{proof}
In the sequel, we compute $\mathbb{E}[|\xX^{}_t(h)|^2]$. Due to \eqref{eq:distancia} 
we start with the computation of the right-hand side of \eqref{eq:distancia}.  
\begin{align*}
\mathbb{E}[|\xX^{(\eta,\e)}_t(h)|^2]&=\mathbb{E}\big[\big\<\xX^{(\eta,\e)}_t(h),\xX^{(\eta,\e)}_t(h)\big\>\big]=
\mathbb{E}\big[\big\<\exp\big(\Theta^{(\eta,\e)}_t\big)h,\exp\big(\Theta^{(\eta,\e)}_t\big)h\big\>\big]\\
&=\mathbb{E}\big[\big\<h,\exp\big(\big(\Theta^{(\eta,\e)}_t\big)^{*}\big)\exp\big(\Theta^{(\eta,\e)}_t\big)h\big\>]\\
&=\mathbb{E}\big[\big\<h,\exp\big(t\cdot(\mathbb{A}^{(\eta,\e)})^*\big)\exp\big((\Lambda^{(\eta,\e)}_t)^*\big)\exp\big(t\cdot\mathbb{A}^{(\eta,\e)}\big)\exp\big(\Lambda^{(\eta,\e)}_t\big)h\big\>\big]\\
&=\mathbb{E}\big[\big\<h,\exp\big(t\cdot(\mathbb{A}^{(\eta,\e)})^*\big)\exp\big((\Lambda^{(\eta,\e)}_t)^*\big)\exp\big(\Lambda^{(\eta,\e)}_t\big)\exp\big(t\cdot\mathbb{A}^{(\eta,\e)}\big)h\big\>\big]\\
&=\big\<h,\exp\big(t\cdot(\mathbb{A}^{(\eta,\e)})^*\big)
\mathbb{E}\big[\exp\big((\Lambda^{(\eta,\e)}_t)^*\big)\exp\big(\Lambda^{(\eta,\e)}_t\big)\big]\exp\big(t\cdot\mathbb{A}^{(\eta,\e)}\big)h\big\>,
\end{align*}
where we use \eqref{e:comutalev}.
By \eqref{e:comutalev} and Campbell's formula  for exponential moments, see for instance \cite{APPLEBAUMBOOK}, we have 
\begin{align}
\mathbb{E}&\big[\exp\big((\Lambda^{(\eta,\e)}_t)^*\big)\exp\big(\Lambda^{(\eta,\e)}_t\big)\big]\nonumber\\
&=
\mathbb{E}\Big[
\exp\Big(\e\int_{0}^{t}\int_{\eta\lqq \|z\|<1}
(\log(I+z^*)+
\log(I+z))\tilde{N}(\ud z, \ud s)
\Big)\Big]\nonumber\\
&=
\exp\Big(\int_{\eta\lqq \|z\|<1}
[\exp(t \e\log(I+z^*)+t \e\log(I+z))-I-
t \e(\log(I+z^*)+\log(I+z))]\nu(\ud z)
\Big)\nonumber\\
&=
\exp\Big( \int_{\eta\lqq \|z\|<1}
[(I+z^*)^{t\e}(I+z)^{t\e}-I-t\e\log((I+z^*)(I+z))]\nu(\ud z)
\Big).\label{eq:expbiendef}
\end{align}
The preceding exponent is well defined due to the series expansion of the logarithm 
\begin{equation}\label{eq:expet}
\begin{split}
(I+z^*)^{t\e}&(I+z)^{t\e}-I-t\e\log((I+z^*)(I+z)) \\
&=
\big(I+t\e z^*+\frac{t\e(t\e-1)}{2!}(z^*)^2+
\frac{t\e(t\e-1)(t\e-2)}{3!}(z^*)^3+\cdots\big)\cdot\\
&\qquad\qquad
\cdot \big(I+t\e z+\frac{t\e(t\e-1)}{2!}z^2+
\frac{t\e(t\e-1)(t\e-2)}{3!}z^3+\cdots\big)\\
&\qquad
-I - t\e\big((z^* + z + z^* z) - \frac{(z^* + z + z^* z)^2}{2}+\frac{(z^* + z + z^* z)^3}{3}-\cdots\big) \\
&=(t^2\e^2-2t\e)(zz^*)+\frac{t\e(t\e-2)}{2}(z^*)^2+\frac{t\e(t\e-2)}{2}z^2+ g(t\e,z,z^*),  
\end{split}
\end{equation}
where 
$g(t\e, z,z^*)$ is such that there exists a positive constant $C(t\e)$  satisfying
\[
\|g(t\e, z,z^*)\|\lqq C(t\e) \|z\|^3 \quad \textrm{for all }\quad \|z\|<1.
\]
For $t\e<1$ the constant $C(t\e)$ can be dominated by $C\cdot t\e$ for some absolute positive constant $C$, that is,
\begin{equation}\label{eq:te1}
\|g(t\e, z,z^*)\|\lqq C\cdot t\e \|z\|^3 \quad \textrm{for all }\quad \|z\|<1.
\end{equation}
Hence there is a constant $C_2>0$ such that 
\[
\|(I+z^*)^{t\e}(I+z)^{t\e}-I-t\e\log((I+z^*)(I+z))\|\lqq C_2 \|z\|^2 \quad \textrm{for all }\quad \|z\|<1.
\]
This guarantees the finiteness of
  \eqref{eq:expbiendef}.
Putting all pieces together, we obtain
\begin{align*}
&\mathbb{E}[|\xX^{(\eta,\e)}_t(h)|^2]\\
&=\Big\<h,\exp\big(t\cdot(\mathbb{A}^{(\eta,\e)})^*\big)
\exp\big( \int_{\eta\lqq \|z\|<1}
[(I+z^*)^{t\e}(I+z)^{t\e}-I-t\e\log((I+z^*)(I+z))]\nu(\ud z)
\big)\exp\big(t\cdot \mathbb{A}^{(\eta,\e)}\big)h\Big\>\\
&=\Big\<\exp\big(t\mathbb{A}^{(\eta,\e)}+\frac{1}{2} \int_{\eta\lqq \|z\|<1}
[(I+z^*)^{t\e}(I+z)^{t\e}-I-t\e\log((I+z^*)(I+z))]\nu(\ud z)
\big)h,\Big.\\
&\qquad\Big.
\exp\big(t\mathbb{A}^{(\eta,\e)}+\frac{1}{2} \int_{\eta\lqq \|z\|<1}
[(I+z^*)^{t\e}(I+z)^{t\e}-I-t\e\log((I+z^*)(I+z))]\nu(\ud z)
\big)h\Big\>\\
&=
|\exp\big(\widetilde{\mathcal{A}}^{(\eta,\e)}_t
\big)h|^2,
\end{align*}
where 
\begin{align*}
\widetilde{\mathcal{A}}^{(\eta,\e)}_t:&=
t\mathbb{A}^{(\eta,\e)}+\frac{1}{2} \int_{\eta\lqq \|z\|<1}
[(I+z^*)^{t\e}(I+z)^{t\e}-I-t\e\log((I+z^*)(I+z))]\nu(\ud z)\\
&=t\big(A+\e\int_{\eta\lqq \|z\|<1} (\log(I+z)-z) \nu(\ud z) \big)\\
&\qquad \qquad+\frac{1}{2} \int_{\eta\lqq \|z\|<1}
[(I+z^*)^{t\e}(I+z)^{t\e}-I-t\e\log((I+z^*)(I+z))]\nu(\ud z)\\
&=tA+\mathcal{D}^{(\e)}_t,
\end{align*}
with
\[
\mathcal{D}^{(\e)}_t:=t\e\int_{\eta\lqq \|z\|<1} (\log(I+z)-z) \nu(\ud z)+\frac{1}{2} \int_{\eta\lqq \|z\|<1}
[(I+z^*)^{t\e}(I+z)^{t\e}-I-t\e\log((I+z^*)(I+z))]\nu(\ud z).
\]
By \eqref{eq:expet} and \eqref{eq:te1} we have  $\|\mathcal{D}^{(\e)}_{t_\e + \vr}\|\to 0$ as $\e \to 0$ for any $\vr$.
Since $\frac{1}{a_\e} = e^{ \la_{N_h}  t_\e}$, 
we obtain
\begin{align*}
\frac{\mathcal{W}_2(\xX^{(\eta,\e)}_{ t_\e + \vr}(h), \mu^\e)}{a_\e}
&=  \frac{1}{a_\e}|\exp(\widetilde \aA^{(\eta,\e)}_{t_\e + \vr})h|=  e^{- \la_{N_h} \vr} |e^{ \la_{N_h} ( t_\e + \vr)} \exp( \widetilde \aA^{(\eta,\e)}_{t_\e + \vr})h|\\
&
=e^{- \la_{N_h} \vr} |\exp(\mathcal{D}^{(\e)}_{t_\e + \vr})e^{ \la_{N_h} ( t_\e + \vr)} \exp(( t_\e + \vr)  A)h|
 \ra e^{- \la_{N_h} \vr} |  v_h|, \quad \textrm{ as }\quad \e \ra 0,
\end{align*}
where in the preceding limit we have used
\eqref{e:jaraheatm09levy} and the limit \eqref{eq:doblelimitelevy}.
\end{proof}

\noindent The following lemma extends Lemma~\ref{lem:exprepresentlevy} passing to the limit $\eta \ra 0$, which can be carried out by a standard stochastic analysis for Poisson random measures. 
\begin{lemma}\label{lem:exprepresentlevyinfint}
We consider $\nu$ being a L\'evy measure satisfying {\eqref{e:comutalev}} and
\begin{equation}\label{eq:zero1}
\int_{0<\|z\|<1} \|z\|^2\nu(\ud z)<\infty \qquad \mbox{ and } \qquad \nu(\{0\}) = 0.
\end{equation}  
Let $(\xX^\e_t)_{t\gqq 0}$ be the unique {mild} solution of
\begin{equation}\label{e:multi1levy0894}
\left\{
\begin{array}{r@{\;=\;}l}
\ud \xX^\e_t &A \xX^\e_t \ud t+  
\e \int_{0< \|z\|<1}(z \xX^{\e}_{t-})\, \widetilde{N}(\ud z,\ud t)  \quad  
\textrm{  for any }\quad t\gqq 0,\\
\xX^\e_0 & I,
\end{array}
\right.
\end{equation}
{in the sense of \eqref{e:mildsolt}.}
Then
we have the following representation
\begin{equation}\label{e:stochexp89}
\xX^{\e}_t(h)=\mathcal{E}^{\e}_{t}h,\qquad {h \in D(A)},\quad t\gqq 0,
\end{equation}
where 
$\mathcal{E}^{\e}_{t}:=\exp(\Theta^{\e}_t)$ and 
\begin{align*}
\Theta^{\e}_t&:=
t\Big(A+\e\int_{0 < \|z\|<1} (\log(I+z)-z) \nu(\ud z) \Big)+\e\int_{0}^{t}\int_{0< \|z\|<1}
\log(I+z)\tilde{N}(\ud z, \ud s)
\end{align*}
{
with $\log(I+z)$ for $\|z\|<1$ is given in \eqref{eq:formalog}.}
\end{lemma}

\noindent
As a consequence of Lemma~\ref{lem:exprepresentlevyinfint} we have the profile cutoff phenomenon for infinite intensity case of $\nu$.
\begin{corollary}\label{cor:infinitoslevy}
We consider $\nu$ being a L\'evy measure satisfying \eqref{e:comutalev} and \eqref{eq:zero1}.
Let $A$ satisfy Hypothesis~\ref{h:sadj} and {$h\in D(A)$}, $h\neq 0$, and consider {the unique mild solution of}
\begin{equation}\label{e:multi1levy}
\left\{
\begin{array}{r@{\;=\;}l}
\ud \xX^\e_t &A \xX^\e_t \ud t+  
\e \int_{0< \|z\|<1}(z \xX^{\e}_{t-})\, \widetilde{N}(\ud z,\ud t)  \quad  
\textrm{  for any }\quad t\gqq 0,\\
\xX^\e_0 & {I},
\end{array}
\right.
\end{equation}
{in the sense of \eqref{e:mildsolt}.}
Furthermore, we assume \eqref{e:comutalev}.
Then for any constant $\vr\in \RR$
and any positive function $(a_\e)_{\e\in (0,1)}$ satisfying 
\eqref{eq:doblelimitelevy},
 we have the profile cutoff phenomenon 
\begin{align*}
\lim_{\e\ra 0} 
\frac{\mathcal{W}_{2}(\xX^\e_{ t_\e + \vr} (h),\mu^\e)}{a_\e} = e^{- \la_{N_h} \vr} | v_h|=:\mathcal{P}_h(\vr),  
\end{align*}
where $ t_\e$ is defined in \eqref{e:tildetepsilonlevy}, and $\la_{N_h}$ and $v_h$ are given in \eqref{e:jaraheatm09levy}.
\end{corollary}
\begin{proof}
The proof follows the lines of the proof of Theorem~\ref{t:multprofilelevy}. 
\end{proof}

\begin{corollary}
Let the assumptions of Corollary~\ref{cor:infinitoslevy} be satisfied.
Then the analogues of Corollary~\ref{cor:infinitos},  Corollary~\ref{cor:errorheatm45}, Corollary~\ref{cor:simpleheatm45} and Corollary~\ref{cor:heatmuchosm45} are valid.
\end{corollary}

\appendix
\section{Proof of the shift linearity of the Wasserstein distance in infinite dimensions}\label{Ap:A1}
\begin{proof}[Proof of Lemma~\ref{lem:properties}, item (d)] Synchronous replica $(U_1,U_1)$ with joint law
{
$\Pi(\ud v_1, \ud v_2)$} (natural coupling) yields the upper bound for any $p>0$ as follows: 
\begin{equation}\label{eq:cotasuperior1}
{
\begin{split}
\mathcal{W}_p(u_1+U_1,U_1)&\lqq \left(\int_{H \times H} |u_1+v_1-v_2|^p\Pi(\ud v_1,\ud v_2)\right)^{\min\{1,1/p\}}\\
&=\left(\int_{\{(v_1,v_2)\in H\times H: v_1=v_2\}} |u_1+v_1-v_2|^p\Pi(\ud v_1,\ud v_2)\right)^{\min\{1,1/p\}}
&=|u_1|^{\min\{1,p\}}.
\end{split}}
\end{equation}
Let $\pi$ be a coupling between $u_1+U_1$ and $U_1$.
Note that
\[
\int_{H \times H}(u-v)\pi(\ud u,\ud v)=\int_{H \times H}u\pi(\ud u,\ud v)-
\int_{H \times H}v\pi(\ud u,\ud v)
=\EE[u_1+U_1]-\EE[U_1]=u_1.
\]
Then
\[
|u_1|=\Big|\int_{H \times H}(u-v)\pi(\ud u,\ud v)\Big|\lqq 
\int_{H \times H}|u-v|\pi(\ud u,\ud v).
\]
Minimizing over all possible couplings between $u_1+U_1$ and $U_1$, we obtain
\begin{equation}\label{eq:cotainferior1}
|u_1|\lqq \mathcal{W}_1(u_1+U_1,U_1).
\end{equation}
For $p\gqq 1$, Jensen's inequality with the help of \eqref{eq:cotasuperior1} and \eqref{eq:cotainferior1}
yields 
\[ 
|u_1|\lqq \mathcal{W}_1(u_1+U_1,U_1)\lqq \mathcal{W}_p(u_1+U_1,U_1)\lqq |u_1|,
\]
and hence $\mathcal{W}_p(u_1+U_1,U_1)=|u_1|$.
For $p\in(0,1)$, the triangle inequality  and the translation invariance (b) of Lemma~\ref{lem:properties} imply
 \begin{align*}
|u_1|^p=\mathcal{W}_p(u_1,0)&\lqq \mathcal{W}_p(u_1,u_1+U)+\mathcal{W}_p(u_1+U,U)+\mathcal{W}_p(U,0)=\mathcal{W}_p(u_1+U,U)+2\EE[|U|^p],
\end{align*}
and hence
\begin{equation}\label{eq:desigualdadinferior}
\mathcal{W}_p(u_1+U,U)\gqq |u_1|^p-2\EE[|U|^p].
\end{equation}
Combining (\ref{eq:cotasuperior1}) and {(\ref{eq:desigualdadinferior})} we obtain (\ref{e:cotauni}). This finishes the proof of item (d).
\end{proof}

\section{Proof of Lemma~\ref{lem:exprepresentlevy}}\label{Ap:A2}
\begin{proof}[Proof of Lemma~\ref{lem:exprepresentlevy}]
Without loss of generality it is enough to show the case $\e=1$. {We drop the dependence on $\e$ and $\eta\in (0,1)$.}
It is enough to show that the right-hand side of \eqref{e:stochexp} defines a solution of \eqref{e:LMN99}. The uniqueness is straightforward. {Recall the definition of $\log(I+z)$ for $\|z\|<1$ given in \eqref{eq:formalog}.}
First we simplify 
\begin{align*}
\Theta_t&=
t\cdot \Big(A -\int_{\eta\lqq \|z\|<1} z \nu(\ud z)\Big)+\int_{0}^{t}\int_{\eta\lqq \|z\|<1}
\log(I+z)N(\ud z, \ud s)\\
&= t\cdot \Big(A -\int_{\eta\lqq \|z\|<1} z \nu(\ud z)\Big)+\sum_{0< s\lqq t} \log(I + (L_s - L_{s-})).
\end{align*}
Next, we introduce the stopping times $T_n := \inf\{t>T_{n-1}~|~(L_t-L_{t-})\neq 0\}$, for $n\in \NN$, and $T_0 = 0$. 
Taking the standard telescopic sums (see proof of Lemma~4.4.6 in~\cite{APPLEBAUMBOOK}), 
\begin{align}\label{e:telescopic}
\exp(\Theta_t) - I
&= \sum_{n =1}^\infty 
\Big(\exp(\Theta_{t\wedge T_n-})-\exp(\Theta_{t\wedge T_{n-1}})\Big) 
+\sum_{n =1}^\infty \Big( \exp(\Theta_{t\wedge T_n})-\exp(\Theta_{t\wedge T_n-})\Big),
\end{align}
however, with a close control of all the commutativity conditions applied.    
We start with the first term on the right-hand side of \eqref{e:telescopic}.  
Note that for $T_{n-1}\lqq s < T_n$ we have 
\begin{align*}
\Theta_s
&= s\cdot \Big(A -\int_{\eta\lqq \|z\|<1} z \nu(\ud z)\Big)+\sum_{m=1}^{n-1} 
\log(I + (L_{T_m} - L_{T_{m}-})).
\end{align*}
In addition, the operator 
\[Q:= A -\int_{\eta\lqq \|z\|<1} z \nu(\ud z)\] 
is the generator of the strongly continuous semigroup $(\exp(s\cdot Q))_{s\gqq 0}$.  
By the commutativity hypothesis \eqref{e:comutalev}, we have that
\[
Q \quad \mbox{ and }\quad \sum_{m=1}^{n-1} 
\log(I + (L_{T_m} - L_{T_{m}-}))
\] 
{commute.}
The linearity of the semigroup yields
{
\begin{align*}
&\exp\Big(\Theta_{t\wedge T_n-}\Big)-\exp\Big(\Theta_{t\wedge T_{n-1}}\Big) \nonumber\\
&=\exp\Big((t\wedge T_n-)\cdot Q \exp\Big(\sum_{m=1}^{n-1} \log(I + (L_{T_m} - L_{T_{m}-}))\Big)\Big)  - 
\exp\Big((t\wedge T_{n-1})\cdot Q \exp\Big(\sum_{m=1}^{n-1} \log(I + (L_{T_m} - L_{T_{m}-}))\Big)\Big)\nonumber\\
&=\Big[\exp\Big((t\wedge T_n-)\cdot Q\Big)  - 
\exp\Big((t\wedge T_{n-1})\cdot Q\Big)\Big] \exp\Big(\sum_{m=1}^{n-1} \log(I + (L_{T_m} - L_{T_{m}-}))\Big)\nonumber\\
&= \int_{t\wedge T_{n-1}}^{t\wedge T_n-} Q 
\exp\Big(s\cdot Q\Big) \ud s \exp\Big(\sum_{m=1}^{n-1} \log(I + (L_{T_m} - L_{T_{m}-}))\Big)\nonumber\\
&= \int_{t\wedge T_{n-1}}^{t\wedge T_n-} Q 
\exp\Big(s\cdot Q\Big) 
\exp\Big(\sum_{m=1}^{n-1} \log(I + (L_{T_m} - L_{T_{m}-}))\Big)
\ud s \nonumber\\
&= \int_{t\wedge T_{n-1}}^{t\wedge T_n-} Q \exp\Big(\Theta_{s}\Big) \ud s = \int_{t\wedge T_{n-1}}^{t\wedge T_n} Q \exp\Big(\Theta_{s}\Big) \ud s.
\end{align*}
}
Since all $T_m < \infty$ a.s., $m\in\NN$, we have 

\begin{equation}\label{e:partecontinua}
\begin{split}
\sum_{n =1}^\infty \Big(\exp(\Theta_{t\wedge T_n-})-\exp(\Theta_{t\wedge T_{n-1}})\Big) 
&= \sum_{n =1}^\infty \int_{t\wedge T_{n-1}}^{t\wedge T_n} Q \exp\Big(\Theta_{s}\Big) \ud s = \int_{0}^{t} \Big(A -\int_{\eta\lqq \|z\|<1} z \nu(\ud z)\Big) \exp\Big(\Theta_{s}\Big) \ud s.
\end{split}
\end{equation}
We continue with the second term on the right-hand side of \eqref{e:telescopic}: 
\begin{align}
&\sum_{n =1}^\infty 
\Big(\exp(\Theta_{t\wedge T_n})-\exp(\Theta_{t\wedge T_{n-}})\Big) \nonumber\\
&= \sum_{n =1}^\infty 
\exp\Big(T_n \cdot \Big(A-\int_{\eta\lqq \|z\|<1} z \nu(\ud z)\Big)
+\sum_{0\lqq m \lqq n} \log(I + (L_{T_m} - L_{T_m-}))\Big)\nonumber\\
&\qquad -\exp\Big((T_n-) \cdot Q+\sum_{0\lqq m < n} \log(I + (L_{T_m} - L_{T_m-}))\Big)
\nonumber\\
&= \sum_{n =1}^\infty 
\exp\Big((T_n-) \cdot Q
+\sum_{0\lqq m < n} \log(I + (L_{T_m} - L_{T_m-}))\Big) \exp
\Big(\log(I + (L_{T_n} - L_{T_n-}))\Big)\nonumber\\
&\qquad -\exp\Big((T_n-) \cdot Q+\sum_{0\lqq m < n} \log(I + (L_{T_m} - L_{T_m-}))\Big)
\nonumber\\
&= \sum_{n =1}^\infty 
\exp\Big((T_n-) \cdot \Big(A-\int_{\eta\lqq \|z\|<1} z \nu(\ud z)\Big)
+\sum_{0\lqq m < n} \log(I + (L_{T_m} - L_{T_m-}))\Big)\nonumber \\
&\qquad\qquad\qquad \cdot
\Big[\exp
\Big(\log(I + (L_{T_n} - L_{T_n-}))\Big)-I\Big]\nonumber\\
&= \sum_{n =1}^\infty
\exp\Big((T_n-) \cdot Q
+\sum_{0\lqq m < n} \log(I + (L_{T_m} - L_{T_m-})\Big) \Big[L_{T_n} - L_{T_n-}\Big]\nonumber\\
&= \int_0^t \int_{\eta \lqq \|z\|<1} z\xX_{s-} N(\ud z,\ud s).\label{e:partediscontinua}
\end{align}
Combining \eqref{e:telescopic} with \eqref{e:partecontinua} and \eqref{e:partediscontinua}, 
we have that $(\mathcal{E}_{t})_{t\gqq 0}$ is a {(strong)} solution of \eqref{e:LMN99}. 
\end{proof}

\section*{Acknowledgments}
\noindent
The research of GB has been supported by the Academy of Finland, via 
the Matter and Materials Profi4 University Profiling Action, 
an Academy project (project No. 339228)
and the Finnish Centre of Excellence in Randomness and STructures (project No. 346306). GB also would like to express
his gratitude to University of Helsinki for all the facilities used along
the realization of this work. 
The research of MAH has been supported by the 
Proyecto INV-2019-84-1837 de la Convocatoria 2020-2021:
``Stochastic dynamics of systems perturbed with small Markovian noise with applications in biophysics, climatology and statistics''
of Facultad de Ciencias at Universidad de los Andes. 

\section*{Statements and Declarations}
\subsection*{Availability of data and material}
Data sharing not applicable to this article as no datasets were generated or analyzed during the current study.
\subsection*{Conflict of interests} The authors declare that they have no conflict of interest.
\subsection*{Authors' contributions}
All authors have contributed equally to the paper.

\end{document}